\documentclass{article}

\usepackage{arxiv}

\usepackage[square,sort,comma,numbers]{natbib}

\usepackage[utf8]{inputenc} 
\usepackage[T1]{fontenc}    
\usepackage{hyperref}       
\usepackage{url}            
\usepackage{booktabs}       
\usepackage{amsfonts}       
\usepackage{nicefrac}       
\usepackage{microtype}      
\usepackage{lipsum}		
\usepackage{graphicx}
\usepackage{doi}
\usepackage{amsthm}
 \usepackage{amsmath}

\usepackage{subcaption}
\usepackage{booktabs}
\usepackage{multirow}

\usepackage[noline, boxed]{algorithm2e}
\SetKwInOut{Parameter}{parameter}

\theoremstyle{plain}
\newcounter{dummy} \numberwithin{dummy}{section}
\newtheorem{theorem}[dummy]{Theorem}

\newtheorem{lemma}[dummy]{Lemma}

\newtheorem{remark}{Remark}

\newtheorem{problem}{Problem}

\usepackage{enumitem}
\newlist{primenumerate}{enumerate}{1}
\setlist[primenumerate,1]{label={4$'$.}}

\theoremstyle{remark}

\newtheorem{example}{Example}

\def\R{\mathbb{R}} 
\def\N{\mathbb{N}} 

\DeclareMathOperator{\Ex}{\mathbb{E}} 
\renewcommand{\Pr}{\mathbb{P}}
\newcommand{\Pm}{P}
\newcommand{\Pmhat}{\widehat{P}}
\newcommand{\Tmhat}{\widehat{T}}

\newcommand{\muhat}{\widehat{\mu}}

\newcommand*\rot{\rotatebox{90}}

\def\sB{\mathcal{B}}
\def\sD{\mathcal{D}}\def\sF{\mathcal{F}}
\def\sG{\mathcal{G}}

\def\sO{\mathcal{O}}
\def\sP{\mathcal{P}}

\def\sX{\mathcal{X}}
\def\sZ{\mathcal{Z}}

\def\bX{{\bf X}}

\def\bZ{{\bf Z}}
\def\bB{{\bf B}}

\def\eps{\varepsilon}
\def\Ind{{\bf 1}}

\newcommand{\Rad}{{\sf Rad}}
\newcommand{\Emp}{{\sf Emp}}
\newcommand{\mt}{{\sf m}}
\newcommand{\qd}{{\sf q}}

\DeclareMathOperator*{\argmax}{arg\,max}
\DeclareMathOperator*{\argmin}{arg\,min}

\newcommand{\SMDetailsVectorMean}{Appendix \ref{sec:sm_10}}
\newcommand{\SMBoundsRegression}{Appendix \ref{sec:sm_11}}
\newcommand{\SMSmallBall}{Appendix \ref{sec:sm_12}}
\newcommand{\SMAlg}{Appendix \ref{sec:sm_13}}
\newcommand{\SMExp}{Appendix \ref{sec:sm_14}}

\newcommand{\MTthmvector}{Theorem \ref{thm:vector}}
\newcommand{\MTsym}{Theorem \ref{thm:symcont}}
\newcommand{\MTtmbounds}{Theorem \ref{thm:uniformTM}}
\newcommand{\MTremarksm}{Remark \ref{rem:restricteps}}
\newcommand{\MTregrelparam}{\S\ref{sub:relevant}}
\newcommand{\MTthmreg}{Theorem \ref{thm:regressionbounds}}
\newcommand{\MTalg}{\S\ref{sec:algorithms}}
\newcommand{\MTsetupB}{\S\ref{sub:setupB}}
\newcommand{\MTthmmaster}{Theorem \ref{thm:tmlinear}}
\newcommand{\MTlemmabounds}{Lemma \ref{lem:bounding}}

\title{Trimmed sample means for robust uniform mean estimation and regression}

\author{ \href{https://orcid.org/0000-0002-1064-3398}{\includegraphics[scale=0.06]{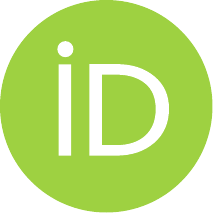}\hspace{1mm}Roberto I. Oliveira}\\
	IMPA\\
	Rio de Janeiro, Brazil \\
	\texttt{rimfo@impa.br} \\
	\And
	\href{https://orcid.org/0000-0001-6188-299X}{\includegraphics[scale=0.06]{orcid.pdf}\hspace{1mm}Lucas Resende} \\
	IMPA\\
    Rio de Janeiro, Brazil\\
	\texttt{lucas.resende@impa.br} \\
}

\renewcommand{\headeright}{}
\renewcommand{\undertitle}{}
\renewcommand{\shorttitle}{Trimmed sample means for robust uniform mean estimation and regression}

\hypersetup{
pdftitle={Trimmed sample means for robust uniform mean estimation and regression},
pdfauthor={Roberto I. Oliveira, Lucas Resende},
pdfkeywords={Sub-Gaussian estimators, robust estimators, regression.},
}

\begin{document}
\maketitle

\begin{abstract}
It is well-known that trimmed sample means are robust against heavy tails and data contamination. This paper analyzes the performance of trimmed means and related methods in two novel contexts. The first one consists of estimating expectations of functions in a given family, with uniform error bounds; this is closely related to the problem of estimating the mean of a random vector under a general norm. The second problem considered is that of regression with quadratic loss. In both cases,  trimmed-mean-based estimators are the first to obtain optimal dependence on the (adversarial) contamination level. Moreover, they also match or improve upon the state of the art in terms of heavy tails. Experiments with synthetic data show that a natural ``trimmed mean linear regression'' method often performs better than both ordinary least squares and alternative methods based on median-of-means.    \end{abstract}


\section{Introduction}

The sample mean is probably the most fundamental way of aggregating information in Statistics \cite{stigler2016seven}. Mathematically, it can be understood as a way to approximate an expected value from a random sample. Aspects of this approximation, including its convergence for large sample sizes, are described by the Law of Large Numbers and various Limit Theorems.

In practice, sample means are often directly to estimate population parameters that correspond to expectations: means, (co)variances and other moments are natural examples. A second use is in $M$-estimation. If a population parameter can be expressed as a minimizer of
\begin{equation}\label{eq:statlearning}L_P(\theta):=\Ex_{X\sim P}\,\ell(X,\theta)\;\;(\theta\in\Theta) \end{equation}
for some suitable function $\ell$, it is natural to estimate this parameter by minimizing 
\begin{equation}\label{eq:statlearningsample}\widehat{L}_n(\theta):=\frac{1}{n}\sum_{i=1}^n\ell(X_i,\theta)\;\;(\theta\in\Theta)\end{equation}
instead, where $(X_1,\dots,X_n)$ is a random sample from $P$. A related setting is ``statistical learning,''~where the goal is to find a near minimizer of $L_P$. One way to do this is to minimize the sample loss (\ref{eq:statlearningsample}) instead: this is the well-known procedure of empirical risk minimization. 

In all of the above cases, sample means are used because as sample-based approximations of the corresponding expectations. However, these are often {\em not} the best such approximations. For instance, a single outlier can change the value of the sample mean arbitrarily. Robust Statistics \cite{huber1965robust,huber2011robust} takes this issue as its starting point, and designs estimators that can withstand a proportion of outliers.  

More subtly, the sample mean is also not optimal in terms of a phenomenon not captured by classical (asymptotic) Robust Statistics: its fluctuations over finite samples can be suboptimal. Consider, for instance, the basic problem of estimating the expectation of a one-dimensional random variable with unit variance. Asymptotically, the sample mean is Gaussian, but Catoni's seminal work \cite{Catoni2012} showed its non-asymptotic bounds are much worse (as Chebyshev's inequality is essentially optimal). What is really striking is that there are other, less obvious estimators with so-called ``sub-Gaussian'' finite-sample error bounds \cite{Catoni2012,Devroye2016,Lee2020}. Some of these estimators can also be made robust to contamination \cite{diakorobust}.

Catoni's discovery led to a surge of interest in finite-sample mean estimation for vectors \cite{Lugosi2019,Lugosi2019b,Lugosi2021,hopkins2020mean,Depersin2022,Minsker2015}, matrices \cite{minsker2018,mendelson2020,abdalla2022} and other objects under (relatively) heavy tails and contamination \cite{Lugosi2019c}, and also on computationally efficient methods with these properties \cite{diakorobust,dong2019,diakonikolas2022}; see \S \ref{sub:related} for more on related work. 

Many of the above papers look at {\em adversarial contamination of the data}. The idea is to obtain estimators that work well even when a small portion of the sample is manipulated arbitrarily. This is a less favorable model for outliers than the Huber contamination model from classical Robust Statistics. We compare and contrast the two models in the end of \S \ref{sub:related}. For now, we simply notice that the sample mean is extremely vulnerable to contamination in either sense.

\subsection{The trimmed mean} In this paper, we show that classical idea of trimmed means allows us to deal both with  heavy-tailed finite samples and with adversarial contamination. 

Suppose one is given $x_{1:n}=(x_1,\dots,x_n)\in \bX^n$ and a function $f:\bX\to\R$. If the $x_i$ are distributed according to a probability distribution $P$ over $\bX$, the sample mean corresponds to the following approximation:
\[Pf = \Ex_{X\sim P}f(X)\approx \frac{1}{n}\sum_{i=1}^nf(x_i),\]
which has the aforementioned problems. By contrast, for an integer $1\leq k<\frac{n}{2}$, the $k$-trimmed mean over $x_{1:n}$ is:
\begin{equation}\label{eq:deftrimmed}
    \widehat{T}_{n,k}(f, x_{1:n}) := \frac{1}{n-2k} \sum_{i=k+1}^{n-k} f(x_{(i)}),
\end{equation}
where $(\cdot)$ is a permutation of $[n]$ such that
\begin{equation*}
    f(x_{(1)}) \leq \cdots \leq f(x_{(n)}).
\end{equation*}
That is, the trimmed mean is the arithmetic mean of the terms that remain once the $k$ largest and $k$ smallest values of $f(x_i)$ have been removed. A large $k$ makes this estimator more robust to outliers, but also introduces some bias. 

The trimmed mean is a classical estimator in Robust Statistics in the sense of Huber  \cite{huber2011robust,stigler2010,huber1972,Stigler1973,Jaeckel1971,Jureckova1981,Hall1981}. More recently, variants of the trimmed mean have been used to estimate high dimensional means \cite{Lugosi2021} and covariances \cite{rico2022a} with non-asymptotic guarantees. The PhD thesis \cite{rico2022} also proves optimality properties of the trimmed mean for estimating the mean of a single function $f$ \cite[Chapter 2]{rico2022}. 

\subsection{This work} The main contribution of the present paper is to show that trimmed means lead to state-of-the-art finite-sample performance in three statistical problems. The first of these is what we call {\em uniform mean estimation:} it is a kind of ``metaproblem'' that appears in several statistical settings. The second problem is the estimation of the mean of a random vector under an arbitrary norm. The third and final problem is that of regression with quadratic loss. 

\subsubsection{Uniform mean estimation} To motivate this problem, we first note that bounding the supremum of an empirical process
\[\sup_{f\in \sF}\left|\frac{1}{n}\sum_{i=1}^nf(X_i) - Pf\right|\]
is a crucial step in the analysis of many problems in Statistics and Machine Learning \cite{geer2000empirical,Boucheron2013}. For instance, one simple way to analyze the empirical risk minimizer -- i.e. the minimizer of the empirical loss in (\ref{eq:statlearningsample}) --, it is standard to show that \begin{equation}\label{eq:statlearninganalysis}L_P(\widehat{\theta}_n) - \inf_{\theta\in\Theta}L_P(\theta)\leq 2\sup_{f\in\sF}\left|\frac{1}{n}\sum_{i=1}^nf(X_i) - Pf\right|\end{equation}
where $\sF=\{\ell(\theta,\cdot)\,:\, \theta\in\Theta\}$.

Now suppose we to replace the sample means in (\ref{eq:statlearningsample}) by some other estimator of $\Ex_{X\sim P}\ell(\theta,X)$. Abstractly, this leads to the problem of bounding 
\[\sup_{f\in \sF}\left|\widehat{E}_f(X_1,\dots,X_n) - Pf\right|\]
where now {\em each $\widehat{E}_f$ should be designed so as to minimize the above supremum}. In other words: given a family of functions $\sF$, we want to design estimators $\widehat{E}_f(X_1,\dots,X_n)$ for each the expectations of $Pf$, $f\in\sF$, so as to minimize the worst-case error. Reference \cite{minsker2018uniform} by Minsker is the only previous paper we are aware of on this problem. 

Our main result on uniform mean estimation is Theorem \ref{thm:uniformTM}, which shows that trimmed means give the best known bounds for this problem in the adversarial contamination setting. In particular, we improve the main result of \cite{minsker2018uniform} and obtain minimax-optimal dependence on moment parameters and the contamination level. Theorem \ref{thm:uniformTM} and some related work are discussed in detail in Section \ref{sec:uniform}.

\subsubsection{Vector mean estimation under arbitrary norms}\label{subsub:vectormean} In this problem, we assume that we have a (potentially corrupted) iid sample $X_{1},\dots,X_n$ from a high-dimensional distribution $P$ over $\R^d$, whose mean $\mu_P$ we want to estimate. The error in our estimate will be measured by an arbitrary norm $\|\cdot\|$. 

As already noted by Minsker \cite{minsker2018uniform}, vector mean estimation is closely related to uniform mean estimation with the function class $\sF$ corresponding to the dual unit ball of the norm $\|\cdot\|$. Using this connection, we present  in Section \ref{sec:vector} a trimmed-mean-based estimator for these vector means. Specifically, we obtain an estimator $\muhat_{n,k}$ whose error $\|\muhat_{n,k}-\mu_P\|$ is bounded by 
\begin{align}C\,\left(\Ex\limits_{X_{1:n}\stackrel{i.i.d.}{\sim}P}\left[\left\|\frac{1}{n}\sum_{i=1}^nX_i - \mu_P\right\|\right] +  \inf_{1\leq q\leq 2}\nu_q\left(\frac{\log(1/\alpha)}{n}\right)^{1-\frac{1}{q}} + \inf_{p>1}\nu_p\eps^{1-\frac{1}{p}}\right)\end{align}
with probability $1-\alpha$, where the $\nu_p$ are one-dimensional moment parameters of the distribution $P$. We will see that this result improves all known results for mean estimation under general norms \cite{Depersin2022,Lugosi2019} and slightly improves the best bounds for the Euclidean norm \cite{Lugosi2021}. Moreover, our bound has minimax-optimal dependence on the contamination level. Unfortunately, the estimator we devise is not computationally efficient. 

\subsubsection{Regression with quadratic risk: theory and heuristics} The third problem we consider is that of finding a function $f\in\sF$ making $\Ex_{(X,Y)\sim P}(Y-f(X))^2$ as small as possible, given a (possibly corrupted) sample from $P$.  

Under suitable technical conditions, we show that a trimmed-mean-based regression method achieves optimal dependence on the contamination level and on moment parameters. This method obtains ``fast''~rates, as in previous work by Lugosi and Mendelson \cite{Lugosi2019b} and Lecué and Lerasle \cite{Lecue2020}, with better dependence on problem parameters. In particular, we once again obtain minimax bounds for the dependence on the contamination level. This result is discussed in Section \ref{sec:regression} along with the related literature. 

Subsequently, we present in Section \ref{sec:experiments} trimmed-mean-based heuristics for robust linear regression. Experiments in the main text and the supplementary material show that our algorithm outperforms a similar method based on the median-of-means principle put forward by \cite{Lecue2020}. These experiments also give us insights on how optimize the performance of these heuristics. 

\subsection{More background}\label{sub:related} Before continuing, we give a general overview of the literature related to our paper. Sections \ref{sec:uniform}, \ref{sec:vector} and \ref{sec:regression} have more specific pointers to the literature. 

Catoni's original breakthrough on mean estimation \cite{Catoni2012} was motivated by his work with Audibert on robust linear regression \cite{audibert2011robust}. Similarly, the literature has focused both on fundamental mean estimation problems and on their applications to various statistical tasks. 

The problem of estimating vector means under the Euclidean norm has attracted much attention. Lugosi and Mendelson \cite{Lugosi2019a} were the first to achieve sub-Gaussian bounds for heavy-tailed finite samples using a method based on the so-called ``median-of-means''~construction. Starting with Hopkins \cite{hopkins2020mean}, a number of papers has presented computationally efficient versions of their method \cite{Cherapanamjeri2020,hopkins2020,Depersin2022}. Lugosi and Mendelson have also put forward a high-dimensional trimmed mean estimator \cite{Lugosi2021}; in particular, our analysis is indebted to their approach. Mean estimation for general norms \cite{Lugosi2019,Depersin2021}; is discussed in Section \ref{sec:vector} below. See also \cite{Lugosi2019c} for a survey of this area.

Robust methods for other statistical  problems have also been considered in various works \cite{audibert2011robust,brownlees2015,mourtada2021,Lecue2020,diakonikolas2019a}. We mention specifically the case of regression with quadratic loss, which was considered in Lugosi and Mendelson \cite{Lugosi2019b} and Lecu\'{e} and Lerasle \cite{Lecue2020} via approaches based on median-of-means. These papers are further discussed in Section \ref{sec:regression}, where we present our own results on this problem.

We briefly mention some additional lines of research. The estimation of covariance matrices is a special case of mean estimator; see \cite{mendelson2020} and the minimax-optimal bounds in \cite{abdalla2022,rico2022a}. Computationally efficient methods for Robust Statistics in high dimensions include the breakthrough by Diakonikolas et al. \cite{diakorobust}, and also \cite{dong2019,diakonikolas2022,hopkins2020} and many other references.  

Finally, we discuss the {\em adversarial contamination model} that we consider along with \cite{dong2019,diakonikolas2022,hopkins2020,diakorobust,Lugosi2021,Depersin2022}. This is a model for outliers in an i.i.d. sample which allows arbitrary changes to a small fraction of sample points before the statistician gets to see it (see \S \ref{sub:corruption} for a formal definition). This is a more demanding model than Huber's classical contamination model \cite{huber1965robust,huber2011robust}, where the uncontaminated data distribution $P$ is usually assumed known, but a small, random fraction of the sample is replaced with another i.i.d. sample from a different (unknown) distribution. There are at least two reasons why studying the adversarial model makes sense. One is that that there settings where the assumption that outliers are i.i.d. from some distribution is not realistic. The second reason is that, when the uncontaminated data distribution $P$ is unknown, then minimax-optimal lower bounds for the adversarial model are achieved (up to constant factors) by Huber-style contamination (see e.g. by Minsker \cite[Lemma 5.4]{minsker2018uniform}). In other words, if we do not know $P$, then the two contamination models are (nearly) equally difficult for the statistician from a minimax perspective.

\subsection{Organization} The remainder of the paper is organized as follows. Notation and terminology are presented in Section \ref{sec:prelim}. Our theoretical results on uniform mean estimation, vector mean estimation and regression are presented in Sections \ref{sec:uniform}, \ref{sec:vector} and \ref{sec:regression} (respectively), in which we also give additional background on these problems. Experiments on a regression heuristic based on trimmed means are summarized in Section \ref{sec:experiments} using algorithms described in Section \ref{sec:algorithms}. The main proof elements for our results are discussed in Section \ref{sec:proofs}. The appendix contains additional technical lemmata and the statements of some results we need. The supplementary material contains additional lemmata and a detailed description of our experiments. 

\section{Notation and terminology} \label{sec:prelim}
\subsection{Basics} $\N=\{1,2,3,\dots\}$ is the set of positive integers. For $n\in\N$, define $[n]:=\{1,2,\dots,n\}$. The cardinality of a finite set $S$ is denoted by $\# S$.

\subsection{Probabilities, expectations and samples} Given a probability space $(\bZ,\sZ,P)$, we write $Z\sim P$ to denote that $Z$ is a random element of $(\bZ,\sZ)$ with distribution $P$. If $f:\bZ\to \R$ is measurable, we use $Pf$, $Pf(Z)$ or $\Ex_{Z\sim P}f(Z)$ to denote the expectation (integral) of $f$ according to $P$. For $p\geq 1$, we write $L^p(P) = L^p(\bZ,\sZ,P)$ for the corresponding $L^p$ space.

Given $n\in\N$, the elements of $\bZ^n$ are denoted by $z_{1:n}=(z_1,z_2,\dots,z_n)$. We write
\[Z_{1:n}=(Z_1,\dots,Z_n)\stackrel{i.i.d.}{\sim}P\]
if the $Z_i$ are independent and identically distributed (i.i.d.) random elements of $(\bZ,\sZ)$ with common law $P$. 

\subsection{Adversarial contamination}\label{sub:corruption} This model comes from the CS literature \cite{diakorobust}. Given $n\in\N$ and $(\bZ,\sZ,P)$ as above, and also a parameter $\varepsilon\in [0,1)$, a random element $Z^{\varepsilon}_{1:n}$ of $\bZ^n$ is a $\varepsilon$-contaminated i.i.d. sample from $P$ if the following condition holds:
\[\mbox{ there exist }Z_{1:n}\stackrel{i.i.d}{\sim}P\mbox{ such that }\# \{i\in[n]\,:\, Z_i^\varepsilon\neq Z_i\}\leq \varepsilon\,n.\]


\subsection{Compatible measures and empirical processes}\label{sub:admissible} We need a technical condition to ensure the various suprema we consider are well defined. Given $p\geq 1$, we say that a probability measure $P$ over $(\bZ,\sZ)$ and a family $\sF$ of $\sZ$-measurable functions from $\bZ$ to $\R$ are $p$-compatible if $\sF\subset L^p(P)$ and there exists a countable subset $\sD\subset \sF$ such that any $f\in\sF$ is the limit of a sequence in $\sD$ that converges pointwise and in $L^p(P)$ norm. For $1$-compatible $\sF$ and $P$, we define the expectation of the empirical process indexed by $\sF$,
\begin{equation}\label{eq:defempirical}\Emp_n(\sF,P):=\Ex\limits_{Z_{1:n}\stackrel{i.i.d.}{\sim}P}\left[\sup_{f\in\sF}\left|\frac{1}{n}\sum_{i=1}^nf(Z_i) - Pf\right|\right]\end{equation}
and the Rademacher complexity,
\begin{equation}\label{eq:defrademacher}\Rad_n(\sF,P):=\Ex\limits_{\substack{ \epsilon_{1:n}\stackrel{i.i.d.}{\sim}U(\{-1,1\}) \\ Z_{1:n}\stackrel{i.i.d.}{\sim}P}}\left[\sup_{f\in\sF}\left|\frac{1}{n}\sum_{i=1}^n \epsilon_i f(Z_i)\right|\right],
\end{equation}
where it is implicit in the expectation that the $Z_{1:n}$ and $\epsilon_{1:n}$ are independent. Definitions (\ref{eq:defempirical}) and (\ref{eq:defrademacher}) are related by the first part of the following classical results:

\begin{theorem}\label{thm:symcont} Assume that a measure $P$ over $(\bX, \sX)$ is 1-compatible with a family of measurable functions $\sG$ from $\bX$ to $\R$.
\begin{enumerate}
\item {\em (Symmetrization and Desymmetrization; \cite[Lemma 11.4]{Boucheron2013})} $\Emp_n(\sG)\leq 2\Rad_n(\sG)$. Moreover, if $Pg=0$ for all $g\in \sG$, then $\Rad_n(\sG)\leq 2\Emp_n(\sG)$.
\item {\em (Ledoux-Talagrand contraction; \cite[Theorem 11.6]{Boucheron2013})} If $\tau:\R\to\R$ is $L$-Lipschitz and satisfies $\tau(0)=0$, and we set $\tau \circ \sG:=\{\tau \circ g\,:\, g\in\sG\}$, then
$\Rad_n(\tau \circ \sG)\leq L\,\Rad_n(\sG)$. 
\end{enumerate}\end{theorem}

\section{Uniform mean estimation via trimmed means}
\label{sec:uniform}

This section discusses the following problem first posed by Minsker \cite{minsker2018uniform}.

\begin{problem}[Uniform mean estimation]\label{problem:uniform} One is given a measurable space $(\bX,\sX)$; and family $\sF$ of measurable functions from $\bX$ to $\R$; and a family $\sP$ of probability distributions over $(\bX,\sX)$ such that $\sF$ and $P$ are $1$-compatible for all $P\in\sP$. For a sample size $n\in\N$; a contamination parameter $\varepsilon\in [0,1/2)$; and a confidence level $1-\alpha\in (0,1)$; the goal is to find a family of measurable functions (estimators) \[\{E_f:\bX^n\to \R\,:\,f\in\sF\}\]
with the following property: for any $P\in\sP$, if $X_{1:n}^\varepsilon$ is a $\varepsilon$-contaminated sample from $P$ (cf. \S \ref{sub:corruption}), then:
\begin{equation}\label{eq:equationproblem1}\Pr\left\{\sup_{f\in\sF}\left|E_f(X_{1:n}^\varepsilon) - Pf\right|\leq \Phi_P\right\}\geq 1-\alpha;\end{equation}
with $\Phi_P=\Phi_P(\sF,\alpha,n,\varepsilon)$ as small as possible.\end{problem}

We assume implicitly that the supremum in the above event is a random variable (i.e. it is a measurable function). Importantly, the estimators $E_f$ are not allowed to depend on the specific measure $P\in\sP$, but may depend on $\sP$, $\sF$, $\alpha$, $n$ and $\varepsilon$. In fact, the dependence on $\alpha$ is unavooidable in our setting even if $\sF$ consists of a single function \cite[Theorem 3.2]{Devroye2016}. On the other hand, it is not known if an optimal mean estimator must depend on $\eps$, and we put this question to one side for the remainder of the paper.  

Minsker \cite{minsker2018uniform} motivates Problem \ref{problem:uniform} via maximum likelihood estimation, and he also applies his result to vector mean estimation (discussed in the next section). Of course, many results exist on this Problem for the case where $E_f$ is the sample mean and $\eps = 0$: this includes Gaussian approximations \cite{vandervaart1996,Chernozhukov2014} and concentration inequalities \cite{talagrand1996,Bousquet2002}. However, the whole point of this paper is that we do {\em not} expect sample means to be optimal estimators, even when there is no contamination.


In what follows, we consider bounds on $\Phi_P(\sF,\alpha,n,\varepsilon)$ in terms of moment conditions and measures of ``complexity''~ of the class $\sF$. For $1$-compatible $\sF$ and $\Pm$ as above, and exponents $p\geq 1$, we define the following (potentially infinite) moment quantities:
\begin{equation}\label{eq:defnup}\nu_p(\sF,P):= \sup_{f\in \sF}(\Pm\,|f-\Pm f|^p)^{\frac{1}{p}} = \sup_{f\in\sF}\|f-Pf\|_{L^p(\Pm)}.\end{equation}
As for our complexity measure of $\sF$ under $\Pm$, we take the expectation of the supremum of the empirical process over an uncontaminated sample (\ref{eq:defempirical}). The question we address is: how small can we expect $\Phi_P(\sF,\alpha,n,\varepsilon)$ to be in terms of the above parameters? The next theorem offers an answer to this question.

\begin{theorem}[Proof in \S \ref{sub:proof:uniformTM}]\label{thm:uniformTM} In the setting of Problem \ref{problem:uniform}, let $\sP$ denote the family of all probability distributions $P$ over $(\bX,\sX)$ that are $1$-compatible with $\sF$ and such that $\Emp_n(\sF,\Pm)<+\infty$. If
\[ \phi := \frac{1}{n} \left(\lfloor \eps n \rfloor + \left\lceil \ln\frac{2}{\alpha} \right\rceil \vee \left\lceil \frac{\left( \frac{1}{2} - \eps \right) \wedge \eps}{2} n \right\rceil\right) < \frac{1}{2},\]
then the family of estimators
\[E_f(x_{1:n}):=\Tmhat_{n,\phi n}(f;x_{1:n})\,\,(x_{1:n}\in\bX^n,\,f\in\sF),\]
satisfies the following property: if $X_{1:n}^\varepsilon$ is an $\eps$-contaminated sample from some $P\in \sP$, then
\begin{equation*}\Pr\left\{\sup_{f\in\sF}\left|\Tmhat_{n,\phi n}(f,X_{1:n}^\varepsilon) - Pf\right|\leq \Phi_P\right\}\geq 1-\alpha,\end{equation*}
where $\Phi_P = \Phi_P(\sF,\alpha,n,\varepsilon)$ is defined as follows:
\[\Phi_P = C_\eps \left(8\Emp_{n}(\sF,\Pm) + \inf_{q \in [1,2]} \nu_q(\sF,\Pm)\left(\frac{ \ln\frac{3}{\alpha}}{n}\right)^{1-\frac{1}{q}} + \inf_{p\geq 1} \nu_p(\sF,\Pm)\eps^{1-\frac{1}{p}}\right),\]
with $C_\eps := 384\left(1 + \frac{\eps}{ \eps \wedge \left(\frac{1}{2}-\eps\right) }\right)$.\end{theorem}

To understand this result, consider first the case where $\sF=\{f\}$ consists of a single function. The value of $\Phi_P$ is (up to constant factors) a sum of two terms: a {\em random fluctuations term} and a {\em contamination term}:
\begin{equation}\label{eq:twoterms}\inf_{1\leq q\leq 2}\nu_p(\{f\},\Pm)\,\left(\frac{1}{n}\ln\frac{1}{\alpha}\right)^{1-\frac{1}{q}}\mbox{ and }\inf_{p\geq 1}\nu_q(\{f\},\Pm)\,\varepsilon^{1-\frac{1}{p}}, \mbox{ respectively.}\end{equation}
The necessity of these two terms follows from \cite[Theorem 3.1]{Devroye2016} and \cite[Lemma 5.4]{minsker2018uniform}\footnote{The lower bound in \cite[Lemma 5.4]{minsker2018uniform} is given for $p\in [2,3]$, but the same proof works for all $p>1$.}. Therefore, the trimmed mean achieves minimax-optimal bounds in the single function case (this is proven by other means in \cite[Chapter 2]{rico2022}). 

It is interesting to contrast this result with what is known about the median of means (MoM) construction \cite{Alon2016,Devroye2016}, which is often used to obtain robust mean estimators. MoM consists of splitting the sample into $K$ parts, taking the sample mean of each part, and then taking the median of the $K$ means. For $K\approx \varepsilon n + \ln(1/\alpha)$, this estimator achieves:
\[\Phi_P(n,\alpha,\varepsilon)= C\,\inf_{1\leq p\leq 2}\nu_p(\{f\},\Pm)\,\left(\frac{1}{n}\ln\frac{1}{\alpha}+\varepsilon\right)^{1-\frac{1}{p}},\]
with $C>0$ universal; this follows e.g. \cite[Lemma 2]{bubeck2013} (for $\varepsilon=0$) combined with the fact that taking $K\gg \varepsilon n$ naturally makes the estimator robust (as observed in e.g. \cite{Lecue2020}). In general, this bound is strictly worse than the optimal (\ref{eq:twoterms}). 

Consider now general function classes $\sF$. To the best of our knowledge, the only paper dealing specifically with Problem \ref{problem:uniform} is \cite{minsker2018uniform} by Minsker. That paper presents an estimator based on a combination of influence functions and median of means. His estimator is optimally robust (i.e. has an optimal contamination term) under $p$-th moment conditions in the range $p\in [2,3]$. Theorem \ref{thm:uniformTM} improves on Minsker's results \cite{minsker2018uniform} in terms of the moment conditions in the bounds. For instance, when $\eps=0$ our bound implies a non-trivial ``sub-Gaussian'' error bound
\[\Phi_P \leq C_\eps \left(8\Emp_{n}(\sF,\Pm) +  \nu_2(\sF,\Pm)\sqrt{\frac{ \ln\frac{3}{\alpha}}{n}}\right)\]
when both $\Emp_{n}(\sF,\Pm)$ and $\nu_2(\sF,\Pm)$ are finite. By contrast, Minsker \cite{minsker2018uniform} requires higher moment assumptions and some knowledge of $\nu_2(\sF,\Pm)$ to obtain a similar bound. These seem to be intrinsic aspects of his analysis, as it relies on Barry-Esseen bounds (which require higher moments) and on a version of Catoni's estimator (which demand some knowledge of the variance). We also go beyond Minsker by obtaining nontrivial results when $\nu_2(\sF,\Pm)=+\infty$, and improved dependence on the contamination level under $p$-th moment assumptions for $p>3$. 

In the next section, we compare previous results on vector mean estimation by several authors, and show that our application of Theorem \ref{thm:uniformTM} in this setting leads to improved results.

\begin{remark}Theorem \ref{thm:uniformTM} has other potential applications besides the ones we discuss at length in this paper. As one example, we mention the problem of estimating probability distributions according to {\em integral probability metrics} \cite{Sriperumbudur2012}. Let $\sF$ be an arbitrary family of measurable functions from $\bX$ to $\R$, and let $\Delta_{\sF}$ denote the set of all probability distributions over $\bX$ that integrate all $f\in\sF$. $\sF$ defines an integral probability (semi)metric over $\Delta_{\sF}$ via the recipe:
\[D_{\sF}(P,Q):=\sup_{f\in\sF}\left|\int_{\bX}\,f(x)\,(Q-P)(dx)\right|\,\,(P,Q\in\Delta_{\sF}).\]
Popular examples of such metrics include the $L^1$ Wassertein metric and kernel-based metrics discussed in \cite{Sriperumbudur2012,Sriperumbudur2016}. Estimating $D_{\sF}(P,Q)$ from contaminated random samples from $P$ and $Q$ can be reduced to Problem \ref{problem:uniform}.\end{remark}

\begin{remark}[Is the complexity term optimal?]\label{rem:complexityuniform} An important theoretical problem is whether the complexity term $\Emp_n(\sF,P)$ that appears in \cite{minsker2018uniform} and our bound is necessary. We only know that this is the case for mean estimation for Gaussian vectors \cite{Depersin2021}. On the other hand, our Theorem follows from a more general result that does not even require $\Emp_n(\sF,\Pm)<+\infty$ or $\nu_p(\sF,\Pm)<+\infty$ for any $p>1$; see Theorem \ref{thm:tmlinear} for details.\end{remark}

\section{Vector mean estimation via trimmed means}\label{sec:vector}
In this section, we apply the results of Section \ref{sec:uniform} to the problem of estimating the mean of a random vector. As we will see, our results are optimal for the Euclidean norm, and improve the state-of-the-art for more general norms. 

\begin{problem}[Vector mean estimation]\label{problem:vector} Let $(\bX,\|\cdot\|)$ be a separable, reflexive Banach space with dual $\bX^*$; $\sX$ is the Borel $\sigma$-field on $\bX$. $\sP$ denotes the family of probability measures over $(\bX,\sX)$ with $P\|\cdot\| = \Ex_{X\sim P}\|X\|<+\infty$, so that for all $P\in \sP$ there exists a unique $\mu_P\in \bX$, called the mean of $P$, such that $f(\mu_P) = \Ex_{X\sim P}f(X)$ for all $f\in \bX^*$. 

For a sample size $n\in\N$; a contamination parameter $\varepsilon\in [0,1/2)$; and a confidence level $1-\alpha\in (0,1)$; the goal is to find a 
function $\muhat:\bX^n\to\bX$ such that, if $X^\eps_{1:n}$ is a $\eps$-contaminated sample from some $P\in \sP$,
\begin{equation}\label{eq:vectormean}\Pr\{\|\muhat(X_{1:n}^\eps) - \mu_P\|\leq \Phi_P(n,\alpha,\varepsilon)\}\geq 1-\alpha,\end{equation}
with $\Phi_P(n,\alpha,\varepsilon)$ as small as possible. \end{problem}

To avoid technicalities, we will interpret the probability (\ref{eq:vectormean}) as an inner probability.  Moreover, we allow out estimator $\muhat$ to depend on $1-\alpha$ and $\eps$, as discussed in the paragraph following Theorem \ref{thm:uniformTM}. 

Our main result on Problem \ref{problem:vector} is the following theorem.

\begin{theorem}\label{thm:vector}In the setting of Problem \ref{problem:vector}, let $\bB^*$ denote the dual unit ball of $\bX$. For $1\leq k< \frac{n}{2}$, define $\muhat_{n,k}:\bX^n\to\bX$ so that
\[\muhat_{n,k}(x_{1:n}) \in {\rm arg }\min_{\mu\in\bX}\,\left(\sup_{f\in \bB^*}\left|\Tmhat_{n,k}(f,x_{1:n})-f(\mu)\right|\right)\,\,(x_{1:n}\in\bX^n).\]
If
\[ \phi := \frac{1}{n} \left(\lfloor \eps n \rfloor + \left\lceil \ln\frac{2}{\alpha} \right\rceil \vee \left\lceil \frac{\left( \frac{1}{2} - \eps \right) \wedge \eps}{2} n \right\rceil\right) < \frac{1}{2},\]
then $\muhat_{n,  \phi n }$ satisfies
\begin{equation*}\Pr\{\|\muhat_{n, \phi n}(X_{1:n}^\eps) - \mu_P\|\leq \Phi_P(n,\alpha,\varepsilon)\}\geq 1-\alpha,\end{equation*}
where 
\[\Phi_P = 2C_\eps \left(8\mathbb{E}\left\|\frac{1}{n}\sum_{i=1}^n X_i - \mu_P \right\| + \inf_{q \in [1,2]} \nu_q\left(\frac{ \ln\frac{3}{\alpha}}{n}\right)^{1-\frac{1}{q}} + \inf_{p\geq 1} \nu_p\eps^{1-\frac{1}{p}}\right),\]
with $C_\eps$ as in in Theorem \ref{thm:uniformTM} and $\nu_p$ is the largest $L^p$ norm of a one-dimensional marginal of $X-\mu_P$:
\[ \nu_p = \nu_p(\bB^*,\Pm) := \sup_{f\in \bB^*}(\Pm\,|f(X-\mu_\Pm)|^{p})^\frac{1}{p}.\]
\end{theorem}

The proof of this result is a straightforward application of Theorem \ref{thm:uniformTM}; see \SMDetailsVectorMean\ for details. We also show in that Section that $\muhat_{n,k}$ is a measurable function when $\bX=\R^d$ for some $d\in \N$.

Theorem \ref{thm:vector} allows us to recover and improve the best known results in the literature. For the case where $(\bX,\|\cdot\|)$ is a Hilbert space and $P\|\cdot\|^2<+\infty$, the estimator by Lugosi and Mendelson \cite[Theorem 2]{Lugosi2021} achieves a slightly worse error rate\footnote{The dependence of their theorem on $\eps$ is worse than shown here, but the bound we present can be deduced from their methods. On the other hand, \cite[Theorem 2]{Lugosi2021} seems to be missing the condition $\eps<1/2$, which is needed for nontrivial bounds.}
\[\Phi^{\rm LM}_P = C_\eps \left(\sqrt{\frac{{\rm tr}(\Sigma_P)}{n}} + \nu_2\sqrt{\frac{1}{n}\ln\frac{3}{\alpha}} + \inf_{p\geq 2} \nu_p\eps^{1-\frac{1}{p}}\right),\]
where $\Sigma_P$ is the covariance matrix (in particular, $\nu_2$ is the largest eigenvalue of $\Sigma_P$). Notice that 
\[\sqrt{\frac{{\rm tr}(\Sigma_P)}{n}}  = \sqrt{\mathbb{E}\left\|\frac{1}{n}\sum_{i=1}^n X_i - \mu_P \right\|^2}\geq \mathbb{E}\left\|\frac{1}{n}\sum_{i=1}^n X_i - \mu_P \right\|\]
by Jensen's inequality. Our estimator also allows for a broader range of $p$ in the contamination term. 

For general Banach spaces, the best result up to this paper was achieved by median-of-means-based estimator of Depersin and Lecu\'{e} \cite{Depersin2021}, which obtains
\[\Phi^{\rm DL}_P = C \left(\mathbb{E}\left\|\frac{1}{n}\sum_{i=1}^n X_i - \mu_P \right\|+ \nu_2\sqrt{\eps + \frac{\ln\frac{3}{\alpha}}{n}}\right)\]
under restrictions on $\eps$ that imply $\eps<1/16$. In contrast to this, our theorem  has better dependence on $\eps$ and the moment parameters $\nu_q$; in fact, our dependence is optimal because of (\ref{eq:twoterms}). 

To the best of our knowledge, Theorem \ref{thm:vector} is also the first result on vector mean estimation that does not require second moments, in either setting considered above. 

\begin{remark}Similarly to \cite{Depersin2021}, the estimator $\muhat_{n,k}$ is the minimizer of a convex function. We believe that the optimization heuuristics from \cite{Depersin2021} could be adapted to compute our estimator. A more interesting question that is beyond the scope of this work is whether our estimator can be provably computed in polynomial time. Computationally efficient estimators based on median-of-means are given in  \cite{hopkins2020mean,Cherapanamjeri2020,Depersin2022}. \end{remark}

\begin{remark}In \cite{Depersin2021}, Depersin and Lecué study the behavior of a median of means estimator for a variant of Problem \ref{problem:vector} where $P$ is symmetrical around $\mu_P$ and no moment assumptions are made. Their assumptions \cite[Assumption 3]{Depersin2021} are suitable for their method, but not for the trimmed mean, which seems to require at least weak moment assumptions. Therefore, we do not pursue this direction in this article.\end{remark}

\section{Quadratic risk minimization}
\label{sec:regression}

This section discusses the problem of regression with quadratic loss as studied by Lugosi and Mendelson \cite{Lugosi2019b} and Lecu\'{e} and Lerasle \cite{Lecue2020}. In this setting, we consider probability measures $P$ over product spaces $\bX\times \R$; and we use $P_\bX$ and $P_{\R}$ to denote the respective marginals.

\begin{problem}[Regression with quadratic loss]\label{problem:regression} One is given a measurable space $(\bX,\sX)$; a convex family $\sF$ of measurable functions from $\bX$ to $\R$; and a family of probability measures $\sP$ over $(\bX\times \R,\sX\times \sB)$ with the following property: for all $P\in\sP$, $\sF$ and $P_{\bX}$ are $2$-compatible; $\sF$ is closed in $L^2(P_{\bX})$; and $P_{\R}$ has a finite second moment. 

In this setup, define (for each $P\in \sP$):
\begin{equation}
    \label{def:fstar}
    f^\star_{P} := \argmin_{f \in \sF} \Ex_{(X,Y) \sim P} \left( f(X) - Y \right)^2.
\end{equation}
For a sample size $n\in\N$, a contamination parameter $\eps\in [0,1/2)$; and a confidence level $1-\alpha\in (0,1)$; 
find a mapping \[ F_n:(\bX\times \R)^n\to \sF\]
with the following property: for any $P\in\sP$, if $Z^\eps_{1:n} := \left\{(X_i^\varepsilon,Y_i^\varepsilon)\right\}_{i\in [n]}$ is an $\varepsilon$-contaminated sample from $P$ (cf. \S \ref{sub:corruption}), then the function $\widehat{f}^\varepsilon_n:=F_n(Z^\eps_{1:n})$ achieves
\begin{equation}\label{eq:equationproblem2}\Pr\left\{ \left\| \widehat{f}^\varepsilon_n - f_{P}^\star \right\|_{L^2(P_{\bX})} \leq \Phi_P\right\}\geq 1-\alpha,\end{equation}
with $\Phi_P:=\Phi_P(\sF,\alpha,n,\varepsilon)$ as small as possible.\end{problem}

Let us make some technical comments about this problem. Firstly, as in the previous sections, we allow $\widehat{f}^\varepsilon_n$ to depend on the contamination level and on the desired confidence level, for the same reasons pointed out right after Theorem \ref{thm:uniformTM}.

Secondly, to avoid delicate measurability issues regarding $F_n$, the probability in (\ref{eq:equationproblem2}) should be interpreted as a inner probability. 

Thirdly, it follows from the fact that $\sF$ is convex and closed in $L^2(P_{\bX})$ that $f^\star_P$ is uniquely defined up to a $P_{\bX}$-null set. With these observations, it is clear that Problem \ref{problem:regression} is well-posed. 

A fourth comment is that we will also consider a variant of Problem \ref{problem:regression} where the goal is to minimize {\em excess risk}. Letting:
\[R_P(f):=P\,(f(X)-Y)^2\,\,(f\in\sF),\]
one can check that
\[\forall f\in\sF\,:\, \left\| f - f^\star_P \right\|_{L^2(P_{\bX})}^2\leq R_P(f) - R_P(f^\star_P) .\]
Therefore, results on the excess risk of $f=\widehat{f}^\eps_n$ also bound the distance between $\widehat{f}^\eps_n$ and $f_P^*$. Our main result on the regression problem, Theorem \ref{thm:regressionbounds} below, will give both types of bounds. In what follows we present the definitions and conditions we need to present our solution to Problem \ref{problem:regression}. The following concrete example will serve to illustrate our discussion. 

\begin{example}[Linear regression with independent errors]\label{example:linear} Let $\bX=\R^d$ with the Borel $\sigma$-field $\sX$. We consider the family of linear functions 
\[\sF:=\{f_{\beta}(\cdot) = \langle \cdot,\beta\rangle\,:\,\beta\in\R^d\}.\]
$\sP$ consists of the family of distributions $P$ such that $\Ex_{(X,Y)\sim P}[\|X\|_2^2+Y^2]<+\infty$ and, given $(X,Y)\sim P$, there exists $\beta_P^\star\in\Theta$ with:
\[Y = \langle \beta_P^\star,X\rangle + \xi_P,\mbox{ where }\xi_P\mbox{ is mean zero and independent from }X.\]
In this case, $f^\star_P(\cdot) = \langle \cdot,\beta_P^{\star}\rangle$ and the excess risk is:
\[R_P(f_\beta) - R_P(f^\star_P) = \|\langle \cdot,\beta -\beta^\star_P\rangle\|_{L^2(P_\bX)}^2 = \langle \beta-\beta^\star_P,\Sigma_P\,(\beta-\beta^\star_P)\rangle,\]
where $\Sigma_P = \Ex_{(X,Y)\sim P}XX^T$ is the population design matrix. We set $\sigma_P^2 = \Ex\,\xi_P^2$.\end{example}

\subsection{Relevant parameters}\label{sub:relevant} As with uniform mean estimation, the analysis of Problem \ref{problem:regression} will depend on moment bounds and complexity parameters of the function class $\sF$. As in previous work, the right parameters to consider are localized \cite{massart2000some,mendelson2015learning,Lugosi2019,Lecue2020}. 

\subsubsection{Some notation}For $f\in\sF$, let \begin{equation}\label{eq:defell}\ell_f(x,y):= (f(x)-y)^2 ((x,y)\in\bX\times \R).\end{equation}
The excess risk of $f$ is
\[R_P(f) - R_P(f_P^{\star}) = P\,(\ell_f-\ell_{f^\star_P}).\]
Now let $\xi_P(x,y):=y-f_P^{\star}(x)$ denote the ``regression residual''~at the optimal $f^\star_P$ (this coincides with the $\xi_P$ in Example \ref{example:linear}), and set:
\begin{equation}\label{eq:multiplierpart}\mt_f(x,y)= \mt_{f,P}(x,y):= \xi_P(x,y)\,(f(x) - f^\star_P(x))\,\,(f\in\sF,\,(x,y)\in \bX\times \R).\end{equation}
Then the difference
\[\ell_f(x,y) - \ell_{f^\star_P}(x,y) = (f(x) - f_P^{\star}(x))^2 - 2\mt_{f,P}(x,y)\]
consists of a ``quadratic term'' and a ``multiplicative term.'' As in \cite{lecue2013learning}, we note that analyzing the standard empirical risk minimizer -- or other risk minimization procedures -- requires {\em lower bounds} on the quadratic part and {\em upper bounds} on the multiplier part. 

\subsubsection{Localization, complexities and critical radii}\label{subsub:localizationparameters} Since Massart \cite{massart2000some} it has been known that local analyses of regression problems lead to the best results. Following Lecué and Mendelson \cite{lecue2013learning}, we consider the local parameters that are relevant to the analysis of quadratic and multiplier parts of our process. What makes these parameters ``local''~is that they are parameterized by the distance to $f^\star_P$. Specifically, define, for $r>0$:
\begin{align*}
    \sF_{\qd}( r, P ) &:= \left\{
    f - f_P^\star \,:\, f\in \sF, ~ \|f-f_P^\star\|_{L^2(P_\bX)} = r \right\},\\
    \sF_{\mt}( r, P ) &:= \left\{  \mt_{f,P} - P\mt_{f,P}  : f \in \sF, ~ \|f-f_P^\star\|_{L^2(P_\bX)}\leq r \right\}.
\end{align*}
As in \cite{Lecue2020}, we use Rademacher complexities (\ref{eq:defrademacher}) to measure the size of these function families. Given constants $\delta_{\qd},\delta_{\mt}>0$, we define the {\em critical radii}
\begin{align*} r_{\qd}(\delta_{\qd},\sF,P)&:=\inf\{r>0\,:\,\sF_{\qd}(r, P)\neq \emptyset\mbox{ and }\Rad_n(\sF_{\qd}(r, P),P)\leq \delta_{\qd}\,r\},\\ 
 r_{\mt}(\delta_{\mt},\sF,P)&:= \inf\{r>0\,:\,\Rad_n(\sF_{\mt}(r, P),P)\leq \delta_{\mt}\,r^2\},
\end{align*}
where (by convention) the infimum of an empty set is $+\infty$. Typically, we will take $\delta_{\qd},\delta_{\mt}<1$. In this case, the intuition is that  $r_{\qd}$ is the smallest radius $r$ at which the quadratic process is significantly larger than $r^2$. The other critical radius $r_{\mt}$ is the smallest radius $r$ at which the multiplier empirical processes becomes small when compared to $r^2$.

\begin{remark}[Critical radii in linear regression with independent errors]\label{rem:linearcomplexity} In Example \ref{example:linear}, one can check that:
\[\Rad_n(\sF_{\qd}(r, P),P)\leq r\,\sqrt{\frac{{\rm tr}(\Sigma_P)}{n}}\mbox{ and }\Rad_n(\sF_{\mt}(r, P),P)\leq r\,\sigma_P\sqrt{\frac{{\rm tr}(\Sigma_P)}{n}}.\]
Therefore, $r_{\qd}(\delta_{\qd},\sF,P)=0$ when $n\geq {\rm tr}(\Sigma_P)/\delta^2_{\qd}$. Moreover,
\[r_{\mt}(\delta_{\mt},\sF,P)\leq \frac{\sigma_P}{\delta_{\mt}}\,\sqrt{\frac{{\rm tr}(\Sigma_P)}{n}}.\] If there is no condition on $n$, it is possible that $r_{\qd}(\delta_{\qd},\sF,P)=+\infty$; for instance, this will be the case if $P_{\bX}$ is Gaussian and $n\leq c\,{\rm tr}(\Sigma_P)/\delta^2_{\qd}$.\end{remark}

\subsubsection{Moment parameters} Our results also require the introduction of two moment-related quantities. The first of these is:
\begin{equation}
    \label{eq:sba}
   \theta_{0}(\sF,P):=\sup\left\{ \frac{\left\|f-f^\star_P\right\|_{L^2(P_\bX)}}{\left\|f-f^\star_P\right\|_{L^1(P_{\bX})}}\,:\, f\in\sF,\, \left\|f-f^\star_P\right\|_{L^1(P_{\bX})}>0\right\}
\end{equation}As shown in Proposition 2 of \cite{lecamLLM}, a bound on $\theta_0(\sF,P)<+\infty$ is essentially equivalent to a ``small ball condition'' on the functions $f-f^\star_P$:
\begin{equation}\label{eq:smallball}\Pr_{X\sim P_{\bX}}\{|f(X)-f^\star_P(X)|\geq c_0\|f-f^\star_P\|_{L^2(P_{\bX})}\}\geq \alpha_0\end{equation}
for $c_0,\alpha_0>0$. This will give a convenient way to control the quadratic part of the excess risk.

The second moment parameter applies to the multiplier part. Given $p \geq 1$, let
\[\kappa_p(\sF,P):=\sup\left\{ \frac{\left\|\mt_{f,P} - P\,\mt_{f,P}\right\|_{L^p(P)}}{\left\|f - f^\star_P\right\|_{L^2(P_{\bX})}}\,:\, f\in\sF,\, \left\|f - f^\star_P\right\|_{L^2(P_{\bX})}>0\right\}.\]

\begin{remark}[Moment conditions and linear regression]\label{remark:momentlinear} In the setting of Example \ref{example:linear}, 
\begin{eqnarray*}\theta_{0}(\sF,P)^{-1} &=&  \inf\limits_{\beta\in\R^d\,:\,\langle \beta,\Sigma_P\,\beta\rangle=1}\|\langle \cdot,\beta\rangle\|_{L^1(P_\bX)}\mbox{ and }\\
\kappa_p(\sF,P) &=&\left(\frac{\|\xi_P\|_{L^p(P)}}{\sigma_P}\right)\, \sup\limits_{\beta\in\R^d\,:\,\langle \beta,\Sigma_P\beta\rangle=1}\|\langle \cdot,\beta\rangle\|_{L^p(P_\bX)}.\end{eqnarray*}
In particular, for $p>2$ the parameter $\kappa_p(\sF,P)$ depends on hypercontractivity properties of the random variable $\xi_P$ and of the one-dimensional marginals of $X\sim P_{\bX}$.\end{remark}

\subsection{Results and comparisons} We present a trimmed-mean-based estimator for Problem \ref{problem:regression} that satisfies improved bounds. Like \cite{Lecue2020}, we use the observation that 
\[f^\star_{P} = \argmin_{f\in\sF}P\,\ell_f = \argmin_{f\in\sF}\left(\sup_{g\in\sF}P\,(\ell_f-\ell_g)\right),\]
and define $F_n$ via
\begin{equation}
    \label{eq:argmin}
    \forall z_{1:n} \in (\bX\times \R)^n\,:\, F_n(z_{1:n})\in\argmin_{f\in\sF}\left(\sup_{g\in\sF}\Tmhat_{n,\phi n}\,(\ell_f-\ell_g,z_{1:n})\right)
\end{equation}
with the convention $\widehat{f}_n^\eps := F_n\left(Z^\eps_{1:n}\right)$.

\begin{theorem}[Proof in \SMBoundsRegression]\label{thm:regressionbounds} In the setting of Problem \ref{problem:regression}, let $\sP$ denote the family of all probability distributions $P$ over $(\bX,\sX)$ that are $2$-compatible with $\sF$.  Given $\alpha \in (0,1)$ and $\eps > 0$ define
\begin{equation}
\label{eq:reghip0} \phi := \frac{\lfloor \eps n \rfloor + \left\lceil \ln\frac{3}{\alpha} \right\rceil \vee \lceil \frac{\eps n}{2} \rceil}{n} \text{ and assume }\phi + \frac{1}{2n}\ln\frac{3}{\alpha} \leq \frac{1}{96\theta_0^2}.
\end{equation}

Then, there is an event $E$ with probability at least $1-\alpha$ where
\begin{equation}
\label{eq:boundtheoremregression}
\left\|\widehat{f}_n^\eps - f^\star_P \right\|_{L^2(P_\bX)} \leq \Phi_P\end{equation}
 with $\Phi_P= \Phi_P(\sF,\alpha,n,\varepsilon)$ given by

\begin{align}\label{eq:defPhiPregression}\Phi_P &:= 49152\,\left[r_\qd\left( \frac{1}{32\theta_0(\sF,P)}, \sF, P \right) \vee 16r_\mt\left( \frac{1}{448\theta_0^2(\sF,P)}, \sF, P \right) \right]\\ \nonumber & \;\;\;\;\; + 49152\,\theta_0(\sF,P)^2\,\left[\inf_{1\leq q\leq 2}\kappa_q(\sF,P)\left(\frac{ \ln\frac{3}{\alpha}  }{n}\right)^{1-\frac{1}{q}}+\inf_{p\geq 1}  \kappa_p(\sF,P)\,\eps^{1-\frac{1}{p}}\right].\end{align}

Moreover, in the same event $E$ the following inequality holds:
\begin{equation}
    \label{eq:oracle}
    R_P\left(\widehat{f}_n^\eps\right) - R_P\left(f^\star_P\right) \leq \left( 1 + \frac{1}{16 \theta_0(\sF,P)^2} \right)\Phi_P^2.
\end{equation}
\end{theorem}

Let us explain how this result compares with previous work on robust regression. 

A seminal early contribution by Audibert and Catoni \cite{audibert2011robust} gives an estimator for robust least squares problems, where $\sF$ is (a subset of) the linear hull of finite family of functions. Unlike our paper, \cite{audibert2011robust} does not consider contamination, requires restricting their regression method to a bounded set, and achieves dimension-dependent bounds. 

Lugosi and Mendelson \cite{Lugosi2019b} gives regressors with ``sub-Gaussian guarantees'' in Problem \ref{problem:regression} under weak moment assumptions, and for more general $\sF$ than \cite{audibert2011robust}. Their method is based on a median-of-means construction. Approaches based on solving min-max problems such as \eqref{eq:argmin} are also present in the literature \cite{lecamLLM, Depersin2021, Lecue2020}. Improving upon \cite{lecamLLM}, Lecu\'{e} and Lerasle \cite{Lecue2020} obtained the best known results for this problem: their bound for $\Phi_P(\sF,\alpha,n,\eps)$ in Problem \ref{problem:regression} takes the following form (up to constant factors):
\begin{equation}\label{eq:lecue2020}\theta_0(\sF,P)^2\left(r_{\qd}(\delta_{\qd},\sF,P)\vee r_{\mt}(\delta_{\mt},\sF,P) + \kappa_2(\sF,P)\sqrt{\varepsilon + \frac{1}{n}\ln\left(\frac{1}{\alpha}\right)}\right),\end{equation}
with smaller values of $\delta_{\mt}$ and $\delta_{\qd}$ than we require. Our bound compares favorably with (\ref{eq:lecue2020}) because it has better dependence on the moment parameters $\kappa_p(\sF,P)$. In particular, Remark \ref{rem:9} shows that the contamination term involving $\eps$ in our theorem is optimal. 

\begin{remark}Lecu\'{e} and Lerasle \cite{Lecue2020} also consider regularized versions of Problem \ref{problem:regression} and do not assume equally distributed random variables. In this paper, we do not consider either generalization. \end{remark}

\begin{remark}\label{rem:restricteps}As in \cite{Lecue2020}, Theorem \ref{thm:regressionbounds} requires a restriction that $\eps\leq c\,\theta_0(\sF,P)^{-2}$ for some positive $c>0$. \SMSmallBall\ explains why some such restriction is necessary for any method, and shows this relates to \S \ref{sub:setupB}.\end{remark}

\begin{remark}\label{rem:9}We also notice that the contamination and the fluctuation terms in Theorem \ref{thm:regressionbounds} are optimal up to constants. To see this recall that if $\mu$ is the mean of a square-integrable random variable $Y$ taking values on $\R$, then
\[ \mu = \argmin_{\mu' \in \R} \Ex(Y-\mu')^2. \]
Thus, in this setup, quadratic risk minimization over the class $\sF$ of constant functions can be used to estimate the mean. Since Theorem \ref{thm:regressionbounds} matches the optimal fluctuation and contamination terms for mean estimation, it must be optimal.
\end{remark}

\section{Algorithms for robust linear regression} \label{sec:algorithms} We now present some heuristics for linear regression that are related to the theoretical results in Theorem \ref{thm:regressionbounds}. Our main finding, presented in \S \ref{sec:experiments}, is that trimmed-mean-based methods tend to outperform ordinary least squares and robust alternatives based on median-of-means. 

Before describing our heuristics, we note that Theorem \ref{thm:regressionbounds} does not translate directly into a practical method. First of all, the definition of $\hat{f}_n^\eps$ requires a choice of confidence level $1-\alpha$, which is (theoretically) unavoidable for any estimator \cite[Theorem 3.2]{Devroye2016}. Theorem \ref{thm:regressionbounds} also requires some knowledge of the contamination level $\eps$, which is also unavoidable for trimmed-mean-based methods. Finally, the min-max problem \eqref{eq:argmin} defining $\hat{f}_n^\eps$ poses computational challenges. 

Nevertheless, we will argue that there are practical ways to choose the trimming parameter $k$ and to compute regressors that circumvent these challenges. This requires several algorithmic choices that we  describe below. 

In our experiments, we compare our method to the following procedures: a variant of the Median of-Means regression procedure of \cite{Lecue2020}; regression with Huber loss; quantile regression for the median quantile; and ordinary least squares. The Supplementary Material discusses these methods in great detail.


\subsection{Preliminaries} We work in the setting of linear regression, corresponding to Example \ref{example:linear} above. For $\beta\in \R^d$, we write $\ell_\beta(x,y):=(\langle \beta,x\rangle -y)^2$ to denote the loss associated with the linear regression function $\langle \beta,\cdot\rangle$ on a point $(x,y)\in \R^d\times\R$. Given a trimming parameter $k\in [n]$; points $z_i=(x_i,y_i)\in\R^d\times \R$, $i\in [n]$; and vectors $\beta^m,\beta^M\in\R^d$, we may write:
\[
\Tmhat_{n,k}\left(\ell_{\beta^m} - \ell_{\beta^M};z_{1:n}\right) =
\frac{1}{n - 2k}
\sum_{i \in I_k(\beta^m, \beta^M, z_{1:n})}
\left( (\langle x_i,\beta^m\rangle - y_i)^2 - (\langle x_i,\beta^M\rangle - y_i)^2 \right),
\]
where $I_k(\beta^m, \beta^M, z_{1:n})$ -- called the {\em active set} for $(\beta^m, \beta^M, z_{1:n})$ -- is the set of indices $i \in [n]$ that appear in the trimmed mean, i.e., the set obtained once the $k$ largest and $k$ smallest values of $(\langle x_i, \beta^m\rangle - y_i)^2 - (\langle x_i, \beta^M\rangle - y_i)^2$ are removed (with ties broken arbitrarily). Our heuristic will be described in two steps. First, we describe how to optimize the choice of $\beta^m$ and $\beta^M$ for a specific $k$. Second, we present a cross-validation procedure to choose the trimming parameter.

\subsection{Optimization for a fixed trimming level} Assume that the trimming parameter is fixed at some value $k$. To find a good regressor for this trimming level, we apply the \textit{Plug-in Method} described in Algorithm \ref{alg:plugin}. It relies on the existence of a black-box function $\texttt{Fit}$ that on input $((x_i,y_i))_{i\in I}$ (with $I\subset [n]$ nonempty) will return a vector 
\[\beta((x_i,y_i)_{i\in I})\in \argmin_{\beta\in\R^d}\,\frac{1}{|I|}\sum_{i\in I}\left(\langle \beta,x_i\rangle - y_i\right)^2.\]
The idea is that the set $I$ will correspond to the set of active indices, of size $(1-2\phi)n = n-2k$. Te algorithm alternates between optimizing $\beta^m$ and $\beta^M$ for a fixed set $I$, and updating $I$ to be the current active set. The output of the algorithm is vector produced in the iterations that minimizes the trimmed mean estimate of the loss. 

In our experiments, we set the number of iterations to $T_{\max}=20$ and implement \texttt{Fit} via the \texttt{lstsq} procedure from the \texttt{numpy.linalg} library in Python. The initial choices for $\beta_0^m, \beta_0^M$ are random perturbations of OLS solutions on the full data; see \SMAlg for details. Before continuing, we note that other optimization heuristics are possible, including variants of the ADMM procedure in \cite{Lecue2020}. As discussed in \SMExp, we choose the Plug-in method because it outperforms the alternatives in experiments.

\begin{algorithm}[ht!]
\SetKwInOut{Input}{input}
\SetKwInOut{Output}{output}
\Input{$(x_1,y_1), \cdots, (x_n,y_n) \in \R^d\times \R$: the data\\$\beta_0^m, \beta_0^M \in \R^d$:  initial guesses \\
$\phi$: trimming ratio\\
$T_{\max}$: number of iterations}

\Output{$\beta^\star_{\phi}$: an approximate solution for the min-max problem}

\BlankLine

$t \leftarrow 0$

$k\leftarrow \lfloor \phi n\rfloor$  (obtain trimming level)

\While{$t<T_{\max}$}{

 $I_t^m \leftarrow I_k\left(\beta^m_t,\beta^M_t ,z_{1:n}\right)$ (get active indices)
    
    $\beta_{t+1}^m \leftarrow \texttt{Fit}\left( \{(x_i,y_i)\}_{i\in I_t^m}\right)$

    $I_t^M \leftarrow I_k\left(\beta^m_{t+1},\beta^M_t ,z_{1:n}\right)$ (update active indices)
    
    $\beta_{t+1}^M \leftarrow \texttt{Fit}\left( \{(x_i,y_i)\}_{i\in I_t^M}\right)$

    $t \leftarrow t+1$

}

$\beta^\star_{\phi}\leftarrow \argmin\left\{\Tmhat_{n,k}(\ell_\beta, z_{1:n})\,:\,\beta \in \bigcup_{t=1}^{T_{\max}}\{\beta^m_t, \beta^M_t\}\right\}.$
\caption{Plug-in algorithm.}\label{alg:plugin}
\end{algorithm}

\subsection{Trimming level selection via cross-validation}\label{subsec:cv} We now consider the problem of choosing the trimming level $k$ for best-possible performance. We here call $\phi \in [0,1]$ a trimming ratio and $k$ a trimming level, both quantities are related by $k = \lfloor \phi n \rfloor$. Following \cite{Lecue2020} we choose $k$ via a cross-validation procedure. Let $(z_i)_{i=1}^n = (x_i,y_i)_{i=1}^n$ be the data points. Assume that $\phi_1\leq \phi_2\leq \dots \leq \phi_m$ is an ordered grid of possible choices for $\phi$. Let $v \leq n$ be the number of folds: that is, $[n]$ is partitioned into sets $\{B_l\}_{l=1}^v$ with respective sizes $n_l := |B_l| \in \{\lfloor n/v\rfloor,\lfloor n/v\rfloor+1\}$. The procedure works as follows. 

\begin{enumerate}
    \item For each choice of $(j,l)\in [m]\times [v]$, let $\beta^\star_{\phi_j}([n] - B_l)$ be the output of Algorithm \ref{alg:plugin} on the $n-n_l$ data points $(x_i,y_i)_{i\not\in B_l}$.
    \item For each choice of $(j,l)\in [m]\times [v]$, estimate the loss of $\beta^\star_{\phi_j}([n] - B_l)$ via a trimmed mean with ratio $\phi_j$ on the fold $B_l$:
    \[L(j,l):=\Tmhat_{n_l, \phi_j n_l} \left(\ell_{f_{j,l}}, (z_{i})_{i\in B_l}\right).\]
    \item For each $j\in [m]$, associate a loss with trimming ratio $\phi_j$ via \[L(j):={\rm median}(L(j,l)\,:\,l\in [v]).\]
    \item Choose $\phi^\star = \phi_{j^\star}$ where $j^\star = \argmax_{j = 2, \dots, m} \frac{L(j-1)}{L(j)}$.
    \item Compute the final estimator $\beta^\star_{\phi^\star}$ by running Algorithm \ref{alg:plugin} with trimming ratio $\phi^\star$ on the full dataset $(z_i)_{i\in[n]}$. 
\end{enumerate}



Intuitively, the choice of trimming ratio $\phi^\star = \phi_{j^\star}$ corresponds to the point at which increasing $j$ induces the largest drop in the loss. This is inspired by the Slope Heuristic \cite{slope}. Another natural choice is to choose $j^\star$ as minimizer of $L(j)$; this method is inspired by \cite{Lecue2020} and considered in the SM. We note in passing that we have fixed $v=5$ as the number of folds in all our experiments. 

\subsection{Other benchmarks} We compare the performance of the trimmed mean estimator against four benchmarks: ordinary least squares, quantile regression, a variant of the Median of Means (MoM) procedure from \cite{Lecue2020}, and Huber regression.

OLS and quantile regression are classical methods. OLS has no parameters, and quantie regression depends only on the quantile level, which we fix at $q=1/2$ (median regression). We rely on the implementation of these methods in \texttt{sklearn}.

MoM requires splitting the $n$ data points into $K$ buckets of approximately equal size. This parameter $K$ is a close analogue of the trimming parameter $k$ in our procedure, and we adapted to MoM the optimization and cross validation procedures described for TM. In brief:

\begin{itemize}
    \item To adapt Algorithm \ref{alg:plugin}, we use a different concept of active set. Consider the blocks $A_1,\dots,A_K\subset [n]$ used by the median-of-means construction. Then the active set $I_k(\beta^m,\beta^M)$ is the block of indices $A_r$ for which
    \[M_r := \frac{1}{\# A_r}\sum_{i\in A_r}(\ell_{\beta^m}(x_i,y_i) - \ell_{\beta^M}(x_i,y_i))\]
    realized the median over all $\{M_s : s \in [K]\}$ with ties between blocks broken arbitrarily.
    \item When performing the optimization iterations, the blocks of median-of-means are resampled uniformly at random at each step.
    \item The cross validation procedure is now over choices of $K_j$. However, the loss estimate $L(j,l)$ are performed via a MoM estimator using the data in fold $B_l$ with $K_j/v$ blocks.
\end{itemize}

These procedures are detailed and contrasted with potential alternatives in \SMExp. Importantly, the original algorithm in \cite{Lecue2020} was designed for sparse linear regression in a $d\gg n$ setting. By contrast, we consider $d\ll n$, and our algorithmic choices for MoM were optimized for this case. This explains why our version of the MoM method is somewhat different from that in \cite{Lecue2020}.

Finally, Huber regression \cite{huber2011robust} is a robust regression method where the loss function interpolates between quadratic error and absolute error. A shape parameter $M \geq 1$ controls this interpolation; we select it  via cross validation from a grid described in the next section. However, instead of using the slope heuristic we take $j^\star = \argmin_{j = 1, \dots, m} L(j)$ (\SMAlg~explains this choice). We use the \texttt{sklearn} implementation of this method.

\section{Experiments with linear regression} \label{sec:experiments} We perform experiments in two data generation setups. Setup A (\S\ref{sub:setupA}) favors robust methods and is analogous to \cite{Lecue2020}. Setup B (\S\ref{sub:setupB}) is closely related to Remark \ref{rem:restricteps} and favors the OLS. In our experiments, the cross-validation procedure (\S\ref{subsec:cv}) uses $v=5$ folds and selects the TM ratio $\phi$ among the values in
\[\Lambda := \left\{ \eps' + \frac{1}{30} : \eps' \in \left\{ 
0, \frac{2}{100}, \frac{4}{100}, \frac{6}{100}, \frac{8}{100}, \frac{10}{100}, \frac{15}{100}, \frac{20}{100}, \frac{30}{100},\frac{40}{100}  \right\} \right\}.\]
The number of buckets $K$ for MoM from the values in $\left\{ \lfloor 2\phi n + 1 \rfloor : \phi \in \Lambda \right\}$. The shape parameter $M$ for Huber regression is selected over $10$ evenly spaced values between $1$ and $1.35$ (Huber \cite{huber2011robust} recommends $M=1.35$ to obtain as much robustness as possible while retaining high statistical efficiency on Gaussian data). Each parameter combination underwent $96$ runs of the experiment.

\subsection{Setup A}\label{sub:setupA} Let $d\geq 1$ be an integer and $X_1, X_2, \cdots, X_n$ be i.i.d. standard Gaussian's. Define $Y_i = \langle X_i , \beta^\star \rangle + \xi_i ~\forall i\in[n]$, where $\beta^\star = \frac{1}{\sqrt{d}}\left[1,1,\cdots,1\right] \in \R^d$ and $\xi_i$ are i.i.d. errors. The distribution of $\xi_i$ follows the skewed generalized t distribution proposed in \cite{lian2025novel}, which has parameters $\mu$, $\sigma$, $\beta$, $\alpha$, and $\lambda$. We vary the parameters $\alpha > 0$ and $\lambda \in (-1,1)$, which control moments and skew, respectively. We set $\mu=0$, $\sigma=1$, and $\beta=2$ so that when $\lambda=0$ the distribution yields the classical Student's t with $2\alpha$ degrees of freedom and the Gaussian taking $\alpha = \infty$. Taking $\lambda \neq 0$ adds skew to these distributions, with direction given by the sign of $\lambda$. The contamination model for Setup A is defined as follows: a set of indices $\sO \subset [n]$ of size $\lfloor\eps n\rfloor$ is chosen uniformly at random, and then one sets
\[ (X_i^\eps, Y_i^\eps) = (X_i, Y_i) ~ \forall i\not\in\sO \text{ and } (X_i^\eps, Y_i^\eps) = (\beta^\star, 10000) ~ \forall i\in\sO. \]

Table \ref{table:setupA} presents our results on four contamination levels. When $\eps > 0$, we not display the results for OLS because they are much larger than other benchmarks. We make the following observations:
\begin{itemize}
    \item In extreme cases where $\xi$ doesn't even have first moment (i.e., $\alpha = 0.5$), quantile regression and MoM perform better; MoM is the best of the two at the largest contamination level we consider.
    \item When $\alpha \geq 1$,  Huber regression and the TM have comparable performance, and TM does better at more skewed distributions.
\end{itemize}

More extensive experiments on variants of Setup A are presented in \SMExp.

\begin{table}
\centering
\caption{Average $\| \widehat{\beta}_n - \beta^\star \|_2$ for several moment and skew configurations.}\label{table:setupA}
\addtolength{\tabcolsep}{-0.11em}
\begin{tabular}{l|l|llll|llll|llll}
\toprule
& & \multicolumn{4}{c}{$\alpha = 0.5$} & \multicolumn{4}{c}{$\alpha = 1$} & \multicolumn{4}{c}{$\alpha = \infty$} \\
& $\lambda$ & 0.0 & 0.3 & 0.6 & 0.9 & 0.0 & 0.3 & 0.6 & 0.9 & 0.0 & 0.3 & 0.6 & 0.9 \\
\midrule

\multirow{5}{*}{ \rot{$\eps = 0$ (no cv)} } & HUBER & 0.47 & 0.49 & 0.60 & 0.86 & \textbf{0.36} & \textbf{0.40} & \textbf{0.50} & \textbf{0.60} & 0.28 & 0.28 & 0.31 & 0.33 \\
& MOM & 3.13 & 3.64 & 4.35 & 5.65 & 0.95 & 1.05 & 1.12 & 1.31 & 0.53 & 0.54 & 0.58 & 0.59 \\
& OLS & 16.03 & 19.23 & 24.11 & 29.44 & 0.79 & 0.86 & 1.05 & 1.25 & \textbf{0.27} & \textbf{0.27} & \textbf{0.29} & \textbf{0.31} \\
& QUANTILE & \textbf{0.45} & \textbf{0.46} & \textbf{0.53} & \textbf{0.77} & 0.39 & 0.48 & 0.66 & 0.85 & 0.33 & 0.34 & 0.38 & 0.42 \\
& TM & 1.33 & 1.46 & 1.83 & 2.44 & 0.47 & 0.48 & 0.54 & 0.64 & 0.29 & 0.28 & 0.29 & 0.31 \\
\midrule

\multirow{4}{*}{ \rot{$\eps = 0.1$} } & HUBER & \textbf{0.52} & 0.54 & 0.63 & \textbf{0.90} & \textbf{0.42} & \textbf{0.47} & 0.64 & 0.82 & 0.33 & 0.33 & 0.36 & 0.39 \\
& MOM & 2.72 & 3.07 & 2.05 & 2.64 & 1.52 & 1.48 & 1.60 & 1.71 & 0.89 & 0.90 & 0.94 & 0.93 \\
& QUANTILE & 0.52 & \textbf{0.53} & \textbf{0.62} & 0.91 & 0.44 & 0.54 & 0.73 & 0.93 & 0.37 & 0.39 & 0.43 & 0.47 \\
& TM & 1.18 & 1.80 & 1.95 & 2.71 & 0.50 & 0.51 & \textbf{0.62} & \textbf{0.72} & \textbf{0.32} & \textbf{0.32} & \textbf{0.33} & \textbf{0.36} \\
\midrule

\multirow{4}{*}{ \rot{$\eps = 0.2$} } & HUBER & 0.61 & 0.64 & 0.76 & \textbf{1.08} & \textbf{0.49} & \textbf{0.54} & 0.72 & 0.93 & 0.40 & 0.38 & 0.42 & 0.46 \\
& MOM & 1.09 & 1.17 & 1.19 & 1.34 & 0.99 & 0.99 & 1.05 & 1.11 & 2.18 & 1.86 & 1.88 & 2.93 \\
& QUANTILE & \textbf{0.61} & \textbf{0.63} & \textbf{0.75} & 1.10 & 0.51 & 0.59 & 0.80 & 1.01 & 0.43 & 0.43 & 0.48 & 0.53 \\
& TM & 0.85 & 0.90 & 1.14 & 1.64 & 0.50 & 0.54 & \textbf{0.65} & \textbf{0.80} & \textbf{0.37} & \textbf{0.37} & \textbf{0.40} & \textbf{0.43} \\
\midrule

\multirow{4}{*}{ \rot{$\eps = 0.4$} } & HUBER & 0.98 & 1.07 & 1.31 & 1.84 & 0.77 & \textbf{0.86} & 1.08 & 1.35 & \textbf{0.62} & \textbf{0.60} & \textbf{0.64} & \textbf{0.70} \\
& MOM & 1.01 & \textbf{1.03} & \textbf{1.01} & \textbf{1.06} & 1.01 & 0.99 & \textbf{0.99} & \textbf{1.02} & 0.97 & 0.98 & 0.97 & 0.98 \\
& QUANTILE & \textbf{0.97} & 1.03 & 1.30 & 1.84 & \textbf{0.76} & 0.86 & 1.14 & 1.41 & 0.63 & 0.61 & 0.68 & 0.74 \\
& TM & 2.04 & 2.29 & 2.58 & 3.50 & 1.01 & 1.13 & 1.34 & 1.78 & 0.74 & 0.68 & 0.72 & 0.80 \\
\bottomrule
\end{tabular}
\end{table}

\begin{figure}[t]
\includegraphics[width=\linewidth]{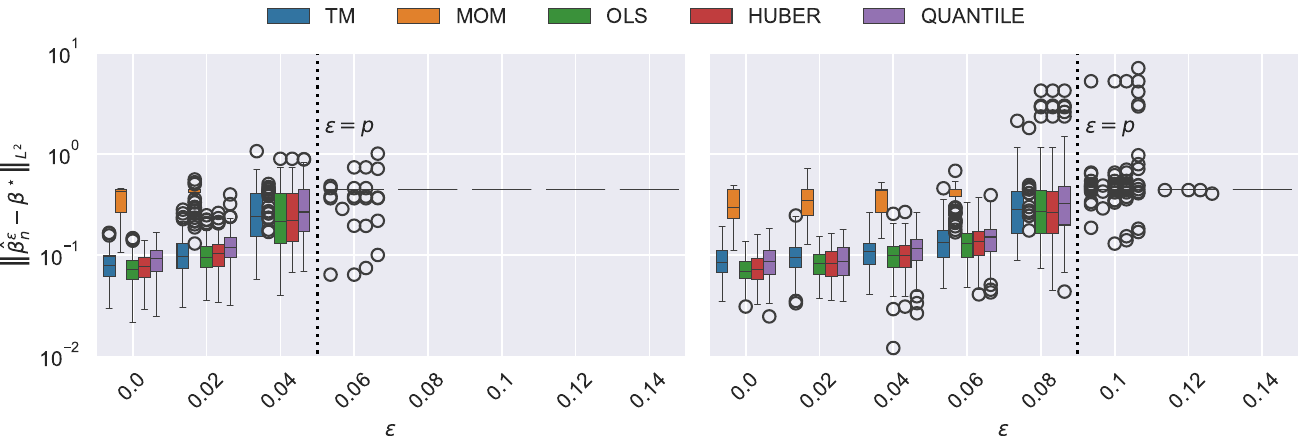}
\caption{Experiments varying the contamination proportion using Setup B, left panel with $p=0.05$ and right panel with $p=0.09$.}
\label{fig:setupB}
\end{figure}

\subsection{Setup B}\label{sub:setupB} This distribution and contamination model are a caricature of missing data and favors OLS. Let $p \in (0,1]$. Take independent \[X_{1:n}'\stackrel{i.i.d.}{\sim} \texttt{Normal}(0_d,I_{d\times d}),\;\xi_{1:n}\stackrel{i.i.d.}{\sim} \texttt{Normal}(0,1)\mbox{ and }B_{1:n}'\stackrel{i.i.d.}{\sim} \texttt{Ber}(p).\]
As before, let $\beta^\star=\frac{1}{d}[1,1,\dots,1]$. The uncontaminated data is given by
\[X_i:= B_i\,\frac{X_i'}{\sqrt{p}}\mbox{ and }Y_i =\langle X_i,\beta^\star\rangle +\xi_i ~ \forall i \in [n].\]
The contamination model is defined as follows. Take a subset $\sO\subset \{i\in [n]\,:\,B_i=1\}$ that satisfies $|\sO|\leq \eps n$ and is as large as possible. Set $B^\eps_i=0$ when $B_i=0$ or $i\in \sO$, and $B^\eps_i=1$ otherwise. The contaminated sample is
\[ X_i^\eps = B^\eps_i\frac{X_i'}{\sqrt{p}}\text{ and }Y_i^\eps = \langle X_i^\eps, \beta^\star \rangle + \xi_i ~ \forall i \in [n]. \]

Figure \ref{fig:setupB} displays the performance of all methods in consideration in Setup B with $d=5$ and $n=1000$. It contrasts with Table \ref{table:setupA} since the OLS is no longer losing to the robust estimators in most cases. Intuitively, OLS ignores points with $B_i^\eps=0$. The performance of TM, OLS, Huber and Quantile are similar, whereas MoM can often be significantly worse. 

\section{Proofs}
\label{sec:proofs}

This section presents the main mathematical arguments in the paper. Some not-too-enlightening details are left to the supplementary material.

\subsection{Trimming and truncation: a master theorem} 

The main results in this paper follow from the ``master theorem''~presented in this section. It may be viewed as an abstract and extended version of the arguments in \cite{Lugosi2021}, which were specific for mean estimation in Hilbert spaces. In this subsection, $P$ is a fixed probability measure over a measurable space $(\bX,\sX)$, and $X_{1:n}^\eps$ is an $\eps$-contaminated sample from $P$. For a given $M>0$, let
\begin{equation}
    \label{eq:deftau}
    \tau_M(x) := \begin{cases}M\text{, if }x>M\\x\text{, if }-M\leq x \leq M\\-M\text{, if }x<-M\end{cases}.
\end{equation}
be called the truncation function. If $\sG$ is a family of functions, we let
\[ \sG^o := \left\{ g - Pg : g \in \sG \right\}, ~ \sG^o_M := \left\{ \tau_M \circ(g - Pg) : g \in \sG \right\} \]
and also ${\rm rem}_M(\sG, P) := \sup_{g \in \sG} |P \tau_{M}\circ g|.$

\begin{theorem}[Master theorem for trimmed mean]\label{thm:tmlinear} Let $m \in \N$ and $\sF_1, \sF_2, \cdots, \sF_m$ be families of functions that are 1-compatible with $P$. Also let $X_{1:n}^\eps$ be an $\eps$-contaminated i.i.d. sample from $P$. Assume that for every $j\in[m]$ there exist $M_j>0$, $b_j\in \{0,1\}$ and $t_j\in \N \cup \{0\}$ satisfying one of the following conditions:
\begin{itemize}
\item either $b_j=0$, $t_j = 0$ and $\sF^o_j$ is a.s. uniformly bounded by $M_j$, i.e. $|f_j - Pf_j| \leq M_j$ almost surely for all $f_j\in\sF_j$;
\item or $b_j=1$ and we have the bound
\begin{equation}\label{eq:Mjhip} \sup_{f_j\in\sF_j} P\left\{\left|f_j(X)-Pf_j\right|>\frac{M_j}{2}\right\} + \frac{8\,\Rad_n(\tau_{M_j} \circ \sF_j^o,P)}{M_j} \leq \frac{t_j}{8n}.\end{equation}
\end{itemize}
Also let $\phi$ be such that
\[\frac{\lfloor \varepsilon n\rfloor + \sum_{j=1}^m t_j}{n}\leq \phi <\frac{1}{2}.\]

Let $x_j \geq 0$ for each $j\in[m]$. Then, with probability at least $1- \sum_{j=1}^m \left(b_j\,e^{-t_j} + e^{-x_j}\right)$, for every linear combination
$f = \sum_{j=1}^m a_j f_j$, $f_j \in \sF_j$,
\begin{align*}
    \left|\Tmhat_{n,\phi n} (f;X_{1:n}^\eps) - Pf\right| & \leq \sum_{j=1}^n |a_j|\left\{2\Emp_n\left( \tau_{M_j}\circ\sF_j^o,P\right) + \eta_j \right\},
\end{align*}
where
\begin{align*}
   \eta_j & = \left(6\phi + \frac{3 x_j}{n}\right)M_j + {\rm rem}_{M_j}\left( \sF_j^o,P\right) +  \nu_2\left( \tau_{M_j}\circ\sF_j^o,P\right)\sqrt{\frac{2x_j}{n}}.
\end{align*}
\end{theorem}


\begin{proof}[Proof of Theorem \ref{thm:tmlinear}] We start with some notation and conventions. For brevity, we will occasionally omit $P$ from our notation; for instance, we write $\Emp_n(\sF)$ instead of $\Emp_n(\sF, P)$. We also let $X_{1:n}\stackrel{i.i.d.}{\sim}P$ be a sample from $P$ with
\[\#\{i\in[n]\,:\,X_i\neq X_i^\eps\}\leq \lfloor \eps n\rfloor.\] 
$\Pmhat_n$ denotes the empirical measure of the clean sample, meaning:
\[\Pmhat_n\,g = \frac{1}{n}\sum_{i=1}^ng(X_i)\mbox{ for }g:\bX\to \R,\]
and we use $\Tmhat^\eps_{n,\phi n}(g):=\Tmhat_{n,\phi n}(g;X_{1:n}^\eps)$ to denote the trimmed mean computed over the contaminated sample. 

The core observation of our proof is that if $k\in\N$ is properly defined and the $M_j$ are big enough, the ``trimmed empirical process'' of linear combinations on a contaminated sample is close to the linear combination of truncated empirical processes on the clean sample, i.e.
\begin{equation}\label{eq:flinearcombapprox} f = \sum_{j=1}^m a_jf_j\mapsto  \widehat{T}^\eps_{n,k} (f) - Pf\approx \sum_{j=1}^m a_j\Pmhat_n\,\tau_{M_j}(f_j(X_{i}) -Pf_j),\end{equation}
which is well understood and satisfies Bernstein-type concentration bounds. This will require two lemmata that we state next. Let $\sG$ be a family of functions and $M>0$, define
\begin{eqnarray}\label{eq:defcountingclean} V_M(\sG) &:=& \sup_{g\in\sG} \sum_{i=1}^n\Ind\{|g(X_i)|>M\},\end{eqnarray}
in words, $V_M(\sG)$ counts the number of large values of $g$ for the worst-case function $g\in \sG$. 

The following ``Counting Lemma'' -- an abstract version of \cite[Lemma 1]{Lugosi2021} -- gives a way to bound the probability that the counting function $V_{M_j}(\sF_j^o)$ exceeds a certain value $t_j$. This will be useful whenever $\sF_j^o$ is not uniformly bounded.

\begin{lemma}[Counting lemma; proof in Appendix \ref{appendix:proofslemmata}]
\label{lem:counting}
Let $\sG$ be a countable family of functions, $t\in\N$ and $n>1$. Assume $M>0$ is such that:
\begin{equation}\label{eq:assumptioncounting}
    \sup_{g\in\sG}P\left\{|g(X)|>\frac{M}{2}\right\} + \frac{8 \Rad_n(\tau_M\circ \sG, P)}{M}\leq \frac{t}{8n}.
\end{equation}
Then $V_M(\sG)\leq t$ with probability at least $1 - e^{-t}$.
\end{lemma}

The next lemma says that if the counting functions $\{V_{M_j}(\sF_j^o)\}_{j \in [m]}$ are all small, then (\ref{eq:flinearcombapprox}) can be justified. 

\begin{lemma}[Bounding lemma; proof in Appendix \ref{appendix:proofslemmata}]
\label{lem:bounding} Let $m\in\N$ and $\sG_1,\dots,\sG_m$ be families of functions from $\bX$ to $\R$. 
Also let $t_1,t_2,\dots,t_m\in\N$ and $M_1,M_2,\dots,M_m\geq 0$ be such that $V_{M_j}( \sG_j ) \leq t_j$ for each $j \in [m]$. If $\phi$ satisfies
\[\frac{\lfloor \varepsilon n\rfloor + \sum_{j=1}^mt_j}{n}\leq \phi <\frac{1}{2},\]
then for any linear combination
\[g:=\sum_{j=1}^ma_jg_j\mbox{ with  }a_j\in\R\mbox{ and }g_j\in\sG_j\mbox{ for each $j\in[m]$},\]
we have
\begin{equation}
    \label{eq:inequalityTM}
    \left|\Tmhat_{n,\phi n}^{\varepsilon}\left(g\right) - \sum_{j=1}^ma_j\Pmhat_n\left(\tau_{M_j}\circ g_j\right)\right| \leq  6\phi \sum_{j=1}^m |a_j|\,M_j.
\end{equation}
\end{lemma}


We come back to the proof of Theorem \ref{thm:tmlinear}. Since the families $\sF_j$ are $1$-compatible with $P$, we may assume they are countable. Our hypotheses ensure that Lemma \ref{lem:counting} can be applied to each $V_{M_j}(\sF_j^o)$ (with the corresponding $t_j$) whenever $b_j=1$. On the other hand, when $b_j=0$, the class $\sF_j^o$ is bounded and $V_{M_j}(\sF_j^o)=0
\leq t_j$ is automatic. Combining this with Lemma \ref{lem:bounding} (taking $\sG_j = \sF^o_j$ for every $j\in[m]$) we have, with probability at least $1 - \sum_{j=1}^mb_j\,e^{-t_j}$, for all $f=\sum_{j=1}^ma_j f_j$,
\[ \left|\Tmhat^\varepsilon_{n,\phi n} (f) - Pf\right| \leq  \left|\sum_{j=1}^m a_j \Pmhat_n \tau_{M_j}(f_j - Pf_j)\right| +  6\phi \sum_{j=1}^m |a_j|\,M_j. \]

We can bound
\[ \left|\sum_{j=1}^m a_j \Pmhat_n \tau_{M_j}(f_j - Pf_j)\right| \leq \sum_{j=1}^m |a_j| \left\{ \sup_{f_j \in \sF_j} \left| (\Pmhat_n-P) \tau_{M_j}(f_j - Pf_j) \right| + {\rm rem}_{M_j}(\sF_j^o,P) \right\}\]

We bound the suprema on the RHS using Bousquet's version of Talagrand's concentration inequality for each class $\sF_j^o$ (Theorem \ref{thm:bousquet}) and observing that
\[\left|\tau_{M_j}(f_j - Pf_j) - P\tau_{M_j}(f_j - Pf_j)\right| \leq 2M_j.\]
We finish using $\sqrt{a+b} \leq \sqrt{a} + \sqrt{b}$ and $\sqrt{ab} \leq a+\frac{b}{4}$ to bound:
\[\sqrt{\frac{8 x_j M_j}{n}  \Emp_n\left(\tau_{M_j}\circ \sF_j^o , P\right)} 
\leq \frac{2x_j}{n}M_j + \Emp_n\left( \tau_{M_j}\circ\sF_j^o, P \right).\]
\end{proof}

\subsection{Bounds for uniform mean estimation}\label{sub:proof:uniformTM}
We now prove our main result on uniform mean estimation, Theorem \ref{thm:uniformTM}. This will require converting the error bounds in Theorem \ref{thm:tmlinear} into moment-based quantities. The next lemma does this; it is somewhat stronger than what we need.  
\begin{lemma}[Proof in the supplementary material] Let $\sF_j$, $t_j$, $b_j$, $x_j$ and $M_j$ satisfy the hypothesis of Theorem \ref{thm:tmlinear}, define
\[ \Bar{\eps} :=  \frac{\left( \frac{1}{2} - \eps \right) \wedge \eps}{1+\sum_{j=1}^m b_j}~,~ t_j = b_j \left(\lceil x_j \rceil \vee \lceil \Bar{\eps} n \rceil\right) \text{ and } \phi := \frac{\lfloor \eps n \rfloor + \sum_{j=1}^m t_j}{n}.\]

Also assume $\phi < \frac{1}{2}$, $x_{j} \geq \frac{1}{3}$ for all $j$ with $b_j=1$ and let 
\[ C_j^\eps := 192\left(1 + \frac{\sum_{l=1}^m t_l}{t_j} + \frac{\eps}{ \bar{\eps} }\right). \]
Then, for every $j$ with $b_j=1$, it is possible to choose $M_{j}$ satisfying (\ref{eq:Mjhip}) such that
\begin{align}\label{eq:sm_bound}
    2\Emp_n\left( \tau_{M_{j}}\circ\sF_{j}^o,P\right) + \eta_{j} &\leq C_j^\eps\left\{8\Emp_n( \sF_{j}, P ) + \inf_{q\in[1,2]}\nu_q( \sF_{j},P ) \left( \frac{x_{j}}{n} \right)^{1-\frac{1}{q}}\right.\\ \nonumber
    &~~~~ + \left.\inf_{p \geq 1}\nu_p( \sF_{j}, P )  \left( \frac{2\eps}{1+\sum_{l=1}^m b_l} \right)^{1 - \frac{1}{p}}\right\}.
\end{align}
\label{lemma:boundeta}
\end{lemma}

\begin{proof}[Proof of Theorem \ref{thm:uniformTM}] Apply Theorem \ref{thm:tmlinear} using $m=1$, $b_1=1$, $x_1:=\ln\frac{3}{\alpha}$ with $t_1$ and $M_1$ as chosen in the Lemma \ref{lemma:boundeta}. Inspection reveals that this leads to the bound claimed in Theorem \ref{thm:uniformTM}.\end{proof}

\begin{appendix}

\section{Proofs of the counting and bounding lemmata}
\label{appendix:proofslemmata}
\begin{proof}[Proof of Lemma \ref{lem:counting}] As in \cite[Lemma 1]{Lugosi2021}, we replace  $V_M(\sG)$ by a \textit{smoother empirical process upper bound} to which we will be able to apply Ledoux-Talagrand contraction and Bousquet's version of Talagrand's concentration inequality. Specifically, we define
\[\eta_M(r):= \left(\frac{2}{M}\left(r-\frac{M}{2}\right)_+\right)\wedge 1\,\,(r\in\R),\]
which is $2/M$-Lipschitz and satisfies 
\[\forall r \geq 0 ~:~ \Ind\{r>M\}\leq \eta_M(r)\leq \Ind\{r>M/2\}.\]
Notice that $\eta_M=\eta_M\circ \tau_M$, a fact that will be useful later. 

To continue, we note that
\begin{equation}\label{eq:varianceexp}\forall g\in\sG\,:\, P\,(\eta_M\circ |g|)^2\leq P\,(\eta_M\circ |g|)\leq P\left\{|g(X)| > \frac{M}{2}\right\}.\end{equation}
One consequence of this is that
\[V_M(\sG) = \sup_{g\in\sG}\sum_{i=1}^n \Ind\{|g(X_i)|>M\}\leq  \sup_{g\in\sG}\sum_{i=1}^n \eta_M\circ |g(X_i)| \leq n\sup_{g\in\sG} P\left\{|g(X)| > \frac{M}{2}\right\}  + nW\]
where $W$ is the empirical process
\[W:=\sup_{g\in\sG}\left|\frac{1}{n}\sum_{i=1}^n\,(\eta_M\circ |g(X_i)| - P\,(\eta_M\circ |g|))\right|.\]
Therefore, 
\[\Pr\left\{V_M(\sG)>t\right\}\leq \Pr\left\{W>\frac{t}{n}-\sup_{g\in\sG}P\left\{|g(X)| > \frac{M}{2}\right\}\right\},\]
and the present lemma will follow once we bound the RHS above by $e^{-t}$. To do so, we apply Bousquet's version of Talagrand's concentration inequality (Theorem \ref{thm:bousquet}) to the class $\eta_M\circ |\sG| = \{ \eta_M(|g|) : g \in \sG \}$ with $C=1$ (as $0\leq \eta_M\leq 1$). Letting $v =  n\sigma^2(\eta_M\circ |\sG|) + 2 \Emp_n(\eta_M\circ |\sG|,P)$ be as in Theorem \ref{thm:bousquet}, we deduce
\[ \Pr\left\{W> \Emp_n(\eta_M\circ |\sG|,P) + \sqrt{\frac{2vt}{n}} + \frac{t}{3n}\right\}\leq e^{-t}.\]
Therefore, we will be done once we show that
\[ t \geq n \sup_{g\in\sG}P\left\{|g(X)| > \frac{M}{2}\right\} + n\,\Emp_n(\eta_M\circ |\sG|,P) + \sqrt{2vtn} + \frac{t}{3}.\]

To prove this last inequality, we bound the empirical process via symmetrization, contraction ($\eta_M$ is $2/M$-Lipschitz), the fact that $\eta_M \circ |\cdot| = \eta_M\circ |\tau_M|$, and our assumption (\ref{eq:assumptioncounting}) relating $t$, $M$ and $n$:
\begin{eqnarray*}
\Emp(\eta_M\circ |\sG|,P) &\leq & 2\Rad_n(\eta_M\circ |\sG|,P) \\
\mbox{($\eta_M\circ |\cdot|=\eta_M\circ |\tau_M|$)}& =&  2\Rad_n(\eta_M\circ |\tau_M\circ \sG|,P) \\
\mbox{(contraction $+ ~\eta_M \circ |\cdot|$ is $2/M$-Lip.)} &\leq &  \frac{4}{M}\,\Rad_n(\tau_M\circ \sG,P)
\end{eqnarray*}

To bound $\sigma^2(\eta_M\circ |\sG|)$, we use (\ref{eq:varianceexp}):
\[\sigma^2(\eta_M\circ |\sG|) = \sup_{g\in\sG}P \,(\eta_M(|g|)-P\eta_M(|g|))^2\leq  \sup_{g\in\sG}P\left\{|g(X)| > \frac{M}{2}\right\}.\]

As a consequence, the variance parameter $v$ in Bousquet's version of Talagrand's concentration inequality can be bounded using (\ref{eq:assumptioncounting}):
\[nv \leq n\sup_{g\in\sG}P\left\{|g(X)| > \frac{M}{2}\right\} + \frac{8n\,\Rad_n(\tau_M\circ \sG,P)}{M} \leq \frac{t}{8}.\]

Combining the above bounds, we arrive at
\[\frac{t}{3} + \sqrt{2vtn} + n\,\Emp_n(\eta_M\circ |\sG|,P) + n\sup_{g\in\sG}P\left\{|g(X)| > \frac{M}{2}\right\} \leq \frac{t}{3} + \frac{t}{2} + \frac{t}{8}\]
which is enough to conclude the proof. \end{proof}

\begin{proof}[Proof of Lemma \ref{lem:bounding}] Define $\tilde{g}_j:=a_jg_j$ and $\tilde{M}_j:=|a_j|\,M_j$ for each $j\in[m]$, so that $g = \sum_{j=1}^m \tilde{g}_j$. We also have
\begin{equation}\label{eq:truncadentroefora}\forall j\in[m]\,:\,\tau_{\tilde{M}_j}\circ \tilde{g}_j = a_j\,\tau_{M_j}\circ g_j.\end{equation}

Our assumption on $V_{M_j}(\sG_j)$ implies that for each $j\in [m]$ there are at most $t_j$ indices such that $|\tilde{g}_j(X_i)|>\tilde{M}_j$. Therefore, the set
\[B^\eps:=\{i\in [n]\,:\, X_i\neq X^\eps_i\mbox{ or }|\tilde{g}_j(X_i)|>\tilde{M}_j\mbox{ for some $j\in [m]$}\}\]
has cardinality bounded by:
\begin{equation}\label{eq:boundBeps}\# B^\eps\leq \lfloor \eps n\rfloor + \sum_{j=1}^m t_j \leq \phi \, n<\frac{n}{2}.\end{equation}
Now define $M:=\sum_{j=1}^m\tilde{M}_j$. Since $g = \sum_{j=1}^m\tilde{g}_j$, we conclude from (\ref{eq:boundBeps}) that there are at most $\phi n$ indices $i\in[n]$ with $|g(X^\eps_i)|>M$. Since the $\phi n$ largest and smallest values of $g(X_i^\eps)$ are excluded from the trimmed mean, we obtain: 
\begin{equation}\label{eq:cantruncateintrimmed}\Tmhat^{\varepsilon}_{n,\phi n}(g)=\Tmhat_{n,\phi n}^{\varepsilon}(\tau_M\circ g).\end{equation}
In what follows, we use this identity to compare the trimmed mean of a sum to a sum of truncated empirical means. More specifically, we let $\Pmhat^\eps_n$ denote the empirical measure of the contaminated sample, and use (\ref{eq:truncadentroefora}) and (\ref{eq:cantruncateintrimmed}) to bound:
\begin{align}\label{eq:boundingtwoterms}\left|\Tmhat_{n,\phi n}^{\varepsilon}\left(g\right) - \sum_{j=1}^ma_j\Pmhat_n\left(\tau_{M_j}\circ g_j\right)\right| & \leq \left|\Tmhat_{n,\phi n}^{\varepsilon}\left(\tau_M\circ g\right) - \Pmhat_n^\eps\,(\tau_M\circ g)\right| \\ \nonumber & + \left|\Pmhat_n^\eps\,(\tau_M\circ g)- \sum_{j=1}^m\Pmhat_n\left(\tau_{\tilde{M}_j}\circ \tilde{g}_j\right)\right|.\end{align}

To bound the first term in the RHS, notice that the empirical mean of $\tau_M(g(X_i^\eps))$ is an average over all sample points in the contaminated sample, whereas 
 $\Tmhat_{n,\phi n}^{\varepsilon}\left(\tau_M\circ g\right)$
 is an average over a subset of these points of size $(1-2\phi)n$. Since all terms in both averages are bounded by $M$ in absolute value, we conclude:
\[\Pmhat^\eps_n\left(\tau_M\circ g\right) = (1-2\phi)\,\Tmhat_{n,\phi n}^{\varepsilon}(\tau_M\circ g) +2\phi \eta\] 
for some $|\eta|\leq M$, from which it follows that:
\begin{equation}\label{eq:trimtruncatehere}\left|\Pmhat^\eps_n\,(\tau_M\circ g) - \Tmhat_{n,\phi n}^{\varepsilon}(\tau_M\circ g)\right|\leq 2\phi\,(|\eta| + |\Tmhat_{n,\phi n}^{\varepsilon}(\tau_M\circ g)|)\leq 4\phi M.\end{equation}
We now consider the difference
\begin{equation}\label{eq:differenceonemultipletruncates}\Pmhat^\eps_n\,(\tau_M\circ g) - \sum_{j=1}^m\Pmhat_n(\tau_{\tilde{M_j}}\circ \tilde{g}_j) = \frac{1}{n}\sum_{i=1}^n \left[\tau_M\circ g(X^\eps_i) - \sum_{j=1}^m\,\tau_{\tilde{M_j}}\circ \tilde{g_j}(X_i)\right]\end{equation}

The $n$ terms inside the square brackets in the RHS are bounded in absolute value by $M + \sum_{j=1}^m\tilde{M}_j = 2M$. Recalling (\ref{eq:boundBeps}), we {\em claim} that if $i\in[n]\backslash B^\eps$, the corresponding term in the RHS of (\ref{eq:differenceonemultipletruncates}) is zero. To see this, fix some $i$ and note that:
\begin{itemize}
\item on the one hand, for each $j\in[m]$, $|\tilde{g}_j(X_i)|\leq \tilde{M}_j$ and so $\tau_{\tilde{M_j}}\circ \tilde{g}_j(X_i)=\tilde{g}_j(X_i)$;
\item on the other hand, since $X_i=X_i^\eps$, and using the above bounds, we have $|g(X_i^\eps)| = |g(X_i)|\leq M$ and \[\tau_M\circ g(X_i^\eps) = g(X_i) = \sum_{j=1}^m\tilde{g}_j(X_i).\]
\end{itemize}
It follows from the claim that:
\[\forall i\in[n]\,:\, \left|\tau_M\circ g(X^\eps_j) - \sum_{j=1}^n\,\tau_{\tilde{M_j}}\circ \tilde{g_j}(X_i)\right|\leq 2M\Ind{\{i\in B^\eps\}},\]
and combining this with (\ref{eq:boundBeps}) and (\ref{eq:differenceonemultipletruncates}) gives:\[\left|\Pmhat^\eps_n\,(\tau_M\circ g) - \sum_{j=1}^m\Pmhat_n(\tau_{\tilde{M_j}}\circ \tilde{g}_j)\right|\leq \frac{2M}{n}\,\# B^\eps\leq 2M\phi\]
The lemma follows from plugging the preceding display together with inequality (\ref{eq:trimtruncatehere}) into the RHS of (\ref{eq:boundingtwoterms}).\end{proof}

\section{Bousquet's version of Talagrand's concentration inequality}
\label{appendix:usefulresults}

We state here the following classical result we use in our proofs. Recall the definitions from \S \ref{sub:admissible}.

\begin{theorem}[Bousquet's version of Talagrand's concentration inequality \cite{Bousquet2002}]\label{thm:bousquet} Assume that a measure $P$ over $(\bX, \sX)$ is 1-compatible with a family of functions $\sG$ and $|g-Pg|\leq C ~ \forall g \in \sG$ for some constant $C>0$. Define:
\[W:= \sup_{g\in\sG} \left|\frac{1}{n}\sum_{i=1}^n\,(g(X_i) - Pg)\right|\mbox{ where }X_{1:n}\stackrel{i.i.d.}{\sim}P\]
Also set $\sigma^2(\sG):=\sup_{g\in\sG}P\,(g-Pg)^2$ and $v := 2C\,\Emp_n(\sG,P) + \sigma^2(\sG)$. Then
\[\forall x>0\,:\,\Pr\left\{W \geq \Emp_n(\sG,P) + \sqrt{\frac{2xv}{n}} + \frac{Cx}{3n}\right\}\leq e^{-x}.\]
\end{theorem}

\section{Moment conditions for the trimmed mean: proof of \MTlemmabounds}\label{sec:sm_9}

\begin{proof}[Proof of \MTlemmabounds]
Our goal is to find, for the values $j$ with $b_j=1$, $M_j$ such that (\ref{eq:Mjhip}) holds and the bound (\ref{eq:sm_bound}) is valid. To start, define
\[ b = \sum_{j=1}^m b_j\text{ and }t =\sum_{j=1}^m t_j. \]

\begin{center}{\bf First step:} choose $M_j$ as a function of $t_j$ and $\sF_j$.\end{center}

Notice that contraction and symmetrization gives
\[ M_j \geq \frac{256\,n}{t_j} \Emp_n(\sF_j) \Rightarrow \frac{8\,\Rad_n(\tau_{M_j} \circ \sF_j^o)}{M_j} \leq \frac{16\,\Emp_n(\sF_j)}{M_j} \leq \frac{t_j}{16 n}. \]

So, (\ref{eq:Mjhip}) follows if we can define
\[M_j = m_j(t_j)\vee \left(\frac{256 n}{t_j}\Emp_n(\sF_j)\right),\]
with $m_j(t_j)$ such that
\[\forall f\in\sF_j\,:\, P\left\{|f(X)-Pf|>\frac{m_j(t_j)}{2}\right\}\leq \frac{t_j}{16n}.\]

We can explicitly find a choice of $m_j(t)$. Markov's inequality gives, for every $p \geq 1$,
\[ \sup_{f\in\sF_j}\mathbb{P}\left(|f(X)-Pf|>\frac{m_j(t_j)}{2}\right) \leq 2^p\frac{\nu_p^p(\sF_j) }{m_j(t_j)^p} \]
and so we can take $m_j(t_j) = 2 \nu_p(\sF_j) \left( \frac{16n}{t_j} \right)^\frac{1}{p}$. Thus, we define
\begin{align}
    \label{eq:defM}
    M_j = M_j(t) &:= \left(2 \nu_p(\sF_j) \left( \frac{16n}{t_j} \right)^\frac{1}{p} \right)\vee \left(\frac{256\,n}{t_j}\Emp_n(\sF_j)\right) \nonumber\\
    &\leq 32 \nu_p(\sF_j) \left( \frac{t_j}{n} \right)^{-\frac{1}{p}} + \frac{256\,n}{t_j}\Emp_n(\sF_j).
\end{align}

\begin{center}{\bf Second step:} bound $2\Emp_n\left( \tau_{M_j}\circ\sF_j^o,P\right) + \eta_j$. \end{center}

Given our choice of $M_j$ we bound the terms of $\eta_j$. We can easily bound
\[{\rm rem}_{M_j}(\sF_j^o) = \sup_{f \in \sF_j} P (f-Pf) \mathbf{1}_{|f - Pf| > M_j} \leq \frac{\nu_p^p(\sF_j)}{M_j^{p-1}} \leq \frac{\nu_p^p(\sF_j)}{m_j(t_j)^{p-1}} \leq \nu_p(\sF_j) \left( \frac{t_j}{n} \right)^{1-\frac{1}{p}} \]
and, using contraction and symmetrization,
\[ \Emp_n\left( \tau_{M_j}\circ\sF_j^o\right) \leq 2\Rad_n\left( \tau_{M_j}\circ\sF_j^o\right) \leq 2\Rad_n\left( \sF_j^o\right) \leq 4 \Emp_n(\sF_j). \]

We now bound the largest variance in $\tau_{M_j} \circ \sF_j^o$ in terms of the moment parameters of $\sF_j$. For $f\in \sF_j$,
\[\mathbb{V}\left(\tau_{M_j}\circ (f-Pf)\right)\leq P\,\left(\tau_{M_j}\circ (f-Pf)\right)^2\leq P\,\left(M_j\wedge |f-Pf|\right)^2,\]
where the first inequality follows from bounding the variance by the second moment, and the second is a consequence of $|\tau_{M_j}\circ (f-Pf)|= M_j\wedge |f-Pf|$. Now, for any $1\leq q\leq 2$,
\[\left(M_j\wedge |f-Pf|\right)^2\leq M_j^{2-q}\,|f-Pf|^q,\]
so that
\[\mathbb{V}\left(\tau_{M_j}\circ (f-Pf)\right)\leq M_j^{2-q}\|f-Pf\|_{L^q(P)}^q\leq M_j^{2-q}\,\nu_q^q(\sF_j).\]

It gives
\[ \nu_2\left( \tau_{M_j}\circ\sF_j^o\right) \sqrt{\frac{2x_j}{n}} \leq \sqrt{2} \left( \frac{M_jx_j}{n} \right)^{1 - \frac{q}{2}} \left( \nu_q(\sF_j) \left(\frac{x_j}{n}\right)^{1-\frac{1}{q}} \right)^\frac{q}{2} \]
and we can bound, using Young's inequality and $\sqrt{2}<2$,
\begin{align*}
    \nu_2\left( \tau_{M_j}\circ\sF_j^o\right) \sqrt{\frac{2x_j}{n}} &\leq \sqrt{2}\left(1-\frac{q}{2}\right)\frac{M_jx_j}{n} + \sqrt{2} \frac{q}{2} \nu_q(\sF_j) \left(\frac{x_j}{n}\right)^{1-\frac{1}{q}} \\
    & \leq 2\left(\frac{M_jx_j}{2n} + \nu_q(\sF_j) \left(\frac{x_j}{n}\right)^{1-\frac{1}{q}}\right).
\end{align*}

Notice that
\[ 6\phi + \frac{4x_j}{n} = \left( \frac{6t}{t_j} + \frac{6\lfloor \eps n\rfloor}{t_j} + \frac{4x_j}{t_j} \right)\frac{t_j}{n} \leq \left(4 + 6\frac{t}{t_j} + 6\frac{\eps}{\bar{\eps} } \right)\frac{t_j}{n}, \]
taking $C'_j = 32\left(4 + 6\frac{t}{t_j} + 6\frac{\eps}{\bar{\eps} } \right)$ and using (\ref{eq:defM}), we have
\[  M_j \left( 6\phi + \frac{4x_j}{n} \right) \leq C'_j\left( \nu_p(\sF_j) \left( \frac{t_j}{n} \right)^{1-\frac{1}{p}} + 8\Emp_n(\sF_j)\right) \]

Combining the bounds gives
\begin{align*}
    2\Emp_n\left( \tau_{M_j}\circ\sF_j^o,P\right) + \eta_j &\leq 8\Emp_n( \sF_j ) + 2\nu_q(\sF_j) \left( \frac{x_j}{n} \right)^{1-\frac{1}{q}}\\
    & ~~~~ + M_j \left( 6\phi + \frac{4x_j}{n} \right) + \nu_p(\sF_j) \left( \frac{t_j}{n}\right)^{1 - \frac{1}{p}} \\
    &\leq 8(1+C_j')\Emp_n( \sF_j ) + 2\nu_q(\sF_j) \left( \frac{x_j}{n} \right)^{1-\frac{1}{q}}\\
    & ~~~~ + (1+C_j')\nu_p(\sF_j) \left( \frac{t_j}{n}\right)^{1 - \frac{1}{p}}
\end{align*}

\begin{center}{\bf Final step:} finish proof by case analysis on $t_j$.\end{center}

Recall
\[C_j^\eps := 192\left(1 + \frac{t}{t_j} + \frac{\eps}{ \bar{\eps} }\right).\]

Since $x_j \geq \frac{1}{3}$ and $q \in [1,2]$ we have
\[ \lceil x_j \rceil^{1-\frac{1}{q}} \leq (x_j+1)^{1-\frac{1}{q}} \leq 2 x_j^{1-\frac{1}{q}}, \]
and we bound
\[
\left( \frac{\lceil x_j \rceil}{n}\right)^{1 - \frac{1}{q}} \leq 2\left( \frac{ x_j }{n}\right)^{1 - \frac{1}{q}}.
\]

We also have $\lceil a \rceil \leq 2a$ when $a \geq 1$. Thus, if $\eps n \geq 1$ (i.e., when there is a contaminated sample point) we have
\[
\frac{\lceil \Bar{\eps} n \rceil}{n} \leq \frac{1}{n} \left\lceil \frac{\eps n}{1+b} \right\rceil \leq \frac{2\eps}{1+b}.
\]
The case $\eps\,n<1$ means that there is no contamination and we might replace $\eps$ with $0$, obtaining the same bound. 

We now take the infimum over $q \in [1,2]$ and $p\geq 1$. If $t_j = \lceil x_j \rceil$,
\[2\Emp_n\left( \tau_{M_j}\circ\sF_j^o,P\right) + \eta_j \leq 8C_j^\eps\Emp_n( \sF_j ) + C_j^\eps \inf_{q \in [1,2]}\nu_q(\sF_j) \left( \frac{ x_j}{n}\right)^{1 - \frac{1}{q}}.\]

The case $t_j = \lceil \Bar{\eps} n \rceil$ gives
\begin{align*}
2\Emp_n\left( \tau_{M_j}\circ\sF_j^o,P\right) + \eta &\leq 8C_j^\eps\Emp_n( \sF_j )\\
& ~~~~ + 4\inf_{q \in [1,2]}\nu_q(\sF_j) \left( \frac{x_j}{n} \right)^{1-\frac{1}{q}} + C_j^\eps \inf_{p\geq 1} \nu_p(\sF_j) \left( \frac{2\eps}{1+b} \right)^{1 - \frac{1}{p}}.
\end{align*}

The final bound follows from considering the two possible values of $t_j$ and performing some overestimates.
\end{proof}

\section{Details on vector mean estimation}\label{sec:sm_10}
In this section we discuss how \MTtmbounds\ can be used to prove \MTthmvector. We also show that our estimator is measurable when $\bX = \R^d$ for some $d \in \N$. Recall that in the context of Problem 2 one has a separable, reflexive Banach space $(\bX, \| \cdot \|)$ with dual $\bX^*$. We also let $P$ be a probability measure such that $P \| \cdot \| = \Ex_{X \sim P}\|X\|<\infty$. This means that the map $f\in \bX^*\mapsto Pf$ is linear and bounded. As a result, there exists a unique $\mu_P \in \bX^{**}=\bX$, called the mean of $P$, such that
\[ f(\mu_P) = Pf \text{ for all }f\in \bX^*. \]

The problem of estimating $\mu_P$ is closely related to the problem of uniform mean estimation over the class $\sF$ given by the dual unit ball of $\bX^*$. To see this let $\hat{E}_f$ be any given estimator for $f \in \sF$, let $\hat{\mu}:\bX ^n\to\bX$ satisfy
\begin{equation}
    \label{eq:strategymean}
    \forall x_{1:n}\in\bX^n\,:\,\hat{\mu}(x_{1:n})\in \argmin_{\mu \in \bX} \left( \sup_{f \in \sF} \left| \hat{E}_f(x_{1:n}) - f(\mu) \right| \right).
\end{equation}
Then, for any given $x_{1:n} \in \bX^n$,
\begin{align*}
    \| \hat{\mu}(x_{1:n}) - \mu_P \| & = \sup_{f \in \sF} \left|f( \hat{\mu}(x_{1:n}) - \mu_P )\right|\mbox{  (because $\sF$ is the dual unit ball)} \\
    & = \sup_{f \in \sF} \left|f( \hat{\mu}(x_{1:n})) - f(\mu_P)\right| \\
    & = \sup_{f \in \sF} \left|f( \hat{\mu}(x_{1:n})) - \hat{E}_f(x_{1:n})\right| + \sup_{f \in \sF}\left| \hat{E}_f(x_{1:n})  - f(\mu_P) \right|\\
    & \leq \sup_{f \in \sF} \left|f( \mu_P) - \hat{E}_f(x_{1:n})\right| + \sup_{f \in \sF}\left| \hat{E}_f(x_{1:n})  - f(\mu_P) \right|\\
    & = 2\sup_{f \in \sF}\left| \hat{E}_f(x_{1:n})  - f(\mu_P) \right|\\
    & = 2\sup_{f \in \sF}\left| \hat{E}_f(x_{1:n})  - Pf \right|
\end{align*}
where in the inequality we used the definition of $\hat{\mu}$ as a minimizer.

On the other hand, given a mean estimator $\hat{\mu}$ for $\mu_P$, we can easily construct  estimators $\hat{E}_f(x_{1:n}) = f(\hat{\mu}(x_{1:n}))$ for each $f\in \sF$, and obtain
\[ \sup_{f \in \sF}\left| \hat{E}_f(x_{1:n})  - f(\mu_P) \right| = \| \hat{\mu}(x_{1:n}) - \mu_P \|. \]

The previous reasoning give us the proof of \MTthmvector.

\begin{proof}[Proof of \MTthmvector]
    Notice that the estimator $\muhat_{n,k}$ is in the form \eqref{eq:strategymean} with $\hat{E}_f$ given by $\Tmhat_{n,k}$. Thus, by the previous reasoning,
    \[ \left\| \muhat_{n,k}(X_{1:n}^\eps) - \mu_P \right\| \leq 2 \sup_{f \in \sF}\left| \Tmhat_{n,k}(f, X^\eps_{1:n})  - Pf \right|. \]
    Selecting $k$ according to \MTtmbounds\ we can bound the RHS, finishing the proof.
\end{proof}

\subsection{Measurability of our estimator when $\bX = \R^d$}

In this subsection we show that our vector mean estimator can be taken to be a measurable function when $\bX = \R^d$. We can write our estimator as
\[ \hat{\mu}_{n,k}(x_{1:n}) \in \argmin_{\mu \in \R^d} \left( \sup_{f \in \sF} \left| \Tmhat_{n,k}(f, x_{1:n}) - f(\mu)  \right| \right), \]
we now must check if it is possible to define $\hat{\mu}_{n,k}(x_{1:n})$ in a way that the estimator is measurable. This is done in the next lemma.

\begin{lemma}Given $1\leq k<n/2$, there exists a measurable function $\hat{\mu}_{n,k}:(\R^d)^n\to  \R$ such that 
 \begin{equation}\label{eq:tmvectormean}\forall x_{1:n}\in(\R^d)^n\,:\,
    \hat{\mu}_{n,k}(x_{1:n}) \in \argmin_{\mu \in \R^d} \left( \sup_{f \in \sF} \left| \Tmhat_{n,k}(f, x_{1:n}) - f(\mu)  \right| \right).
\end{equation}
\end{lemma}
\begin{proof} Recall the dual unit ball $\sF$ is a symmetric class of functions, thus
\[ \sup_{f \in \sF}\, \left|\Tmhat_{n,k}(f, x_{1:n}) - f(\mu)\right| = \sup_{f \in \sF}\, \Tmhat_{n,k}(f, x_{1:n}) - f(\mu). \]

Define
    \[ F(\mu, x_{1:n}) = \sup_{f \in \sF}\, \Tmhat_{n,k}(f, x_{1:n}) - f(\mu), \]
    which is convex in $\mu$ for fixed $x_{1:n}$ because it is a supremum of affine functions. $F$ is also a measurable function of $(\mu,x_{1:n})$ because the supremum can be taken over a countable dense subset $\sD\subset \sF$. For fixed $x_{1:n}$, the values $\Tmhat_{n,k}(f, x_{1:n})$ are uniformly bounded, and one can deduce from this that the set
    \[K(x_{1:n}):=\argmin_{\mu \in \R^d} F(\mu,x_{1:n})\]
    is convex, compact (the dimension is finite), and nonempty. 

    It remains to show that we can take a measurable function $\hat{\mu}_{n,k}$ with $\hat{\mu}_{n,k}(x_{1:n})\in K(x_{1:n})$. In order to use Kuratowski and Ryll-Nardzewski measurable selection theorem \cite{kuratowski1965general} we need to show that for every open set $U \subset \R^d$, the set
    \[ A_U := \{ x_{1:n} : K(x_{1:n}) \cap U \neq \emptyset \} \]
    is measurable. If $U=\emptyset$ we are done. Otherwise, we can write $U = \bigcup_{i=1}^n K_i$ where $K_i$ is compact and has non-empty interior for every $i\in\N$. Thus, it suffices to show that
    \[ A_K := \{ x_{1:n} : K(x_{1:n}) \cap K \neq \emptyset \} \]
    is measurable for every compact set $K$ with non-empty interior. Let $D$ be dense and countable in $\R^d$, and also assume that $K\cap D$ is dense in $K$. Notice that $K(x_{1:n})\cap K \neq \emptyset$ if there is some minimizer of $F(\cdot, x_{1:n})$ in $K$. For a given $x_{1:n} \in (\R^d)^n$,
    \begin{center}
       $K(x_{1:n})\cap K \neq \emptyset$ \\
       $\Leftrightarrow$ \\
       $\exists\mu \in K\text{ such that }F(\mu, x_{1:n}) \leq F(\mu', x_{1:n})\, \forall \mu' \in \R^d$\\
       $\Leftrightarrow$ \\
       $\exists \mu \in K\text{ such that }F(\mu, x_{1:n}) \leq F(\mu', x_{1:n})\,\, \forall \mu' \in D  $\\
       $\Leftrightarrow$ \\
       $\forall m \in \N,\,\exists \mu \in K\cap D\text{ such that }F(\mu, x_{1:n}) \leq F(\mu', x_{1:n}) + \frac{1}{m}\,\, \forall \mu' \in D  $.
    \end{center}
    The first equivalence follows from our previous observation. The second equivalence follows noticing that $F$ is continuous on $\mu$. The last equivalence is less obvious and uses that $K$ is compact:
    \begin{itemize}
        \item ($\Rightarrow$): assume $\mu \in K$ is such that $F(\mu, x_{1:n}) \leq F(\mu', x_{1:n})\,\, \forall \mu' \in D$; since $K\cap D$ is dense in $K$ there exists $(\mu_k)_{k=1}^\infty \subset K \cap D$ satisfying $\mu_k \to \mu$ as $k\to\infty$. Since $F$ is continuous on $\mu$ we have $F(\mu_k, x_{1:n}) \to F(\mu, x_{1:n})$ as $k\to\infty$, given $m \in \N$ exists $k_m$ such that $F(\mu_{k_m}, x_{1:n}) \leq F(\mu, x_{1:n}) + \frac{1}{m}$ and so $F(\mu_{k_m}, x_{1:n}) \leq F(\mu', x_{1:n}) + \frac{1}{m}\,\, \forall \mu' \in D $.

        \item $(\Leftarrow)$: Take a sequence $(\mu_m)_{m=1}^\infty \subset K \cap D$ satisfying $F(\mu_m, x_{1:n}) \leq F(\mu', x_{1:n}) + \frac{1}{m}\,\, \forall \mu' \in D$. Since $K$ is compact there is a sub-sequence $(\mu_{k_m})_{m=1}^\infty \subset (\mu_m)_{m=1}^\infty$ such that $\mu_{k_m} \to \mu$ as $m\to\infty$ for some $\mu \in K$. Moreover, $F(\mu, x_{1:n}) = \lim_{m\to\infty} F(\mu_{k_m}, x_{1:n})$ and so $F(\mu, x_{1:n}) \leq F(\mu', x_{1:n})\,\, \forall \mu' \in D$.
    \end{itemize} 
    
    Writing the previous observations using the appropriate set operation for each logical quantifier yields
    \[ A_K = \bigcap_{m=1}^\infty \bigcup_{\mu \in K \cap D} \bigcap_{\mu' \in D}\left\{ x_{1:n} : F(\mu, x_{1:n}) \leq F(\mu', x_{1:n}) + \frac{1}{m} \right\} \]
    and so $A_K$ is measurable.
\end{proof}

\section{Bounds for regression}\label{sec:sm_11}
We prove our bounds for regression in this section. The probability measure $P$ will often be implicit throughout the section. We will sometimes use the following property, 
\begin{equation}\label{eq:1storderopt}\forall f\in \sF\,:\, P\,\mt_{f}  = P\, \xi_P(X,Y)(f(X) - f^\star_P(X))\leq 0,\end{equation}
where $\mt_f =\mt_{f,P}$ is the ``multiplier term'' (see \MTregrelparam). Indeed, (\ref{eq:1storderopt}) is the first-order optimality condition for $f^\star_P= {\rm arg}\min_{f\in \sF}\,P\,(Y-f(X))^2$. 

In what follows, $Z_{1:n}^\eps$ is an $\eps$-contaminated sample from $P$. Similarly to the proof of \MTthmmaster, we write $\Tmhat_{n,k}^\eps(\cdot):=\Tmhat_{n,k}(\cdot,Z_{1:n}^\eps)$. The estimator $\widehat{f}_n^\eps \in \sF$ of $f^\star_P \in \sF$ is obtained solving the minimization problem
\[ \widehat{f}_n^\eps \in \argmin_{f\in\sF}\left(\sup_{g\in\sG}\Tmhat^\eps_{n,\phi n}\,(\ell_f-\ell_g)\right). \]

The next Lemma reduces \MTthmreg\ to proving localized upper and lower bounds on certain trimmed means. Introduce the notation:
\[r_f := \left\|f - f^\star_P\right\|_{L_2(P_\bX)}.\]

\begin{lemma}
\label{lemma:fixedpointargument}If $r > 0$ and $\gamma > 0$ are such that
\begin{equation}\label{eq:assumpfixedpoint}\inf_{r_f = r} \Tmhat_{n, \phi n}^\varepsilon \left( \ell_f - \ell_{f^\star_P}\right) \geq \gamma \geq 2\,\sup_{r_f \leq r} \Tmhat_{n, \phi n}^\varepsilon \left(\mt_{{f}} - P\mt_{{f}}\right),\end{equation}
then $r_{\widehat{f}_n^\eps}=\left\|\widehat{f}_n^\eps - f^\star_P \right\|_{L^2(P_\bX)} \leq r$ and $R(\widehat{f}_n^\eps)-R(f^\star_P)\leq r^2 + 2\gamma$.\end{lemma}

\begin{proof} The proof proceeds in three stages that follow the ``localization + fixed point'' outline from previous work \cite{mendelson2015learning,Lecue2020,chinot2020robust}. In the first stage, we use a localization argument and show: 
\begin{equation}\label{eq:localizepqsim}\forall f\in \sF\,:\,\Tmhat_{n, \phi n}^\varepsilon \left( \ell_f - \ell_{f^\star_P} \right)\left\{\begin{array}{ll}>\gamma, & r_f>r;\\ \geq -\gamma, & r_f\leq r.\end{array}\right.\end{equation}
In the second stage, we bound $r_{\hat{f}_n^\eps}$ via (\ref{eq:localizepqsim}) and a ``fixed point'' argument. In the final stage, we notice that an excess risk bound is implicit in the first two steps.\\

\noindent \textbf{Localization:} our goal here is to prove (\ref{eq:localizepqsim}). Since $\sF$ is convex, we can scale down elements $f \in \sF$ with $r_f > r$ to elements $\Bar{f} \in \sF$ with $r_{\Bar{f}} = r$ and bound the corresponding trimmed mean via our assumption. Explicitly, assume $q_f = \frac{r_f}{r}>1$ and define $\Bar{f} = f^\star_P + \frac{f-f^\star_P}{q_f}$, so that $\bar{f}\in \sF$ by convexity and $r_{\bar{f}}=r$. We obtain
\begin{align*}
    \Tmhat_{n, \phi n}^\varepsilon \left( \ell_f - \ell_{f^\star_P} \right) &= \,\Tmhat_{n, \phi n}^\varepsilon \left( q_f^2\,(\bar{f}-f^\star_P)^2 - 2\mt_{\bar{f}}\right)\\
    & \geq q_f\, \Tmhat_{n, \phi n}^\varepsilon \left( \ell_{\Bar{f}} - \ell_{f^\star_P}\right) > \gamma \mbox{ using (\ref{eq:assumpfixedpoint}) and $q_f>1$.}
\end{align*}
This proves (\ref{eq:localizepqsim}) for $r_f>r$. For $r_f\leq r$, we use (\ref{eq:1storderopt}) and $(f-f^\star_P)^2\geq 0$ to obtain
\begin{align*}
    \Tmhat^\varepsilon_{n,\phi n}\left( \ell_{f} - \ell_{f^\star_P} \right) & = \Tmhat^\varepsilon_{n,\phi n}\left( (f-f^\star_P)^2 - 2 \mt_f \right)\\ &  \geq \Tmhat^\varepsilon_{n,\phi n}\left( (f-f^\star_P)^2 - 2 \left(\mt_f - P\mt_f\right) \right)\\
    &\geq -2\Tmhat^\varepsilon_{n,\phi n}\left(\mt_f - P\mt_f\right) \geq -\gamma \mbox{  by (\ref{eq:assumpfixedpoint}).}
\end{align*}

\noindent\textbf{Fixed point argument:} for any $f \in \sF$, let 
\[ \Delta(f) = \sup_{g \in \sF} \Tmhat^\varepsilon_{n,\phi n}\left( \ell_f - \ell_g \right),\]
so that $\hat{f}^\eps_n$ minimizes $\Delta$ over $\sF$. Therefore,
\begin{equation}\label{eq:localizefhat}\Tmhat^\varepsilon_{n,\phi n}\left( \ell_{\widehat{f}^\eps_n} - \ell_{f^\star_P} \right)\leq \Delta(\widehat{f}^\eps_n)\leq \Delta(f^\star_P)\end{equation}
The bounds in (\ref{eq:localizepqsim}) show that for any $g\in \sF$:
\[\Tmhat^\varepsilon_{n,\phi n}(\ell_{f^\star_P}-\ell_g)\leq \left\{\begin{array}{ll}\gamma,& r_g\leq r;\\ -\gamma, & r_g>r.\end{array}\right. \]
Since $\gamma>0$, we obtain $\Delta(f^\star_P)\leq \gamma$. This implies via (\ref{eq:localizefhat}) that 
$\Tmhat^\varepsilon_{n,\phi n}\left( \ell_{\widehat{f}^\eps_n} - \ell_{f^\star_P} \right)\leq \gamma$. We deduce that $r_{\widehat{f}^\eps_n}\leq r$ via assumption (\ref{eq:assumpfixedpoint}).

\noindent\textbf{Excess risk:} we have:
\[R_P(\hat{f}^\eps_n) - R_P(f^\star_P) = r_{\hat{f}_n^\eps}^2 - 2\,P\,\mt_{\hat{f}_n^\eps}.\]
The first term in the RHS is $\leq r^2$ by the above. The second can be bounded by:
\begin{align*}
    - 2P\mt_{\widehat{f}_n^\eps} &= 2 \Tmhat^\varepsilon_{n,\phi n}\left( \mt_{\widehat{f}_n^\eps} - P\mt_{\widehat{f}_n^\eps} \right) - 2 \Tmhat^\varepsilon_{n,\phi n}\left( \mt_{\widehat{f}_n^\eps} \right)\\
    & \leq 2 \Tmhat^\varepsilon_{n,\phi n}\left( \mt_{\widehat{f}_n^\eps} -P\mt_{\widehat{f}_n^\eps}\right) +  \Tmhat^\varepsilon_{n,\phi n}\left(\ell_{\hat{f}_n^\eps}-\ell_{f^\star_P}\right) \\ & \leq \gamma + \Delta(\hat{f}_n^\eps)\leq 2\gamma,
\end{align*}
where the last line follows from (\ref{eq:assumpfixedpoint}) combined with $r_{\widehat{f}^\eps_n}\leq r$ and the calculations in the previous step.\end{proof}

We now apply Lemma \ref{lemma:fixedpointargument} to prove our main result on regression. 

\begin{proof}[Proof of \MTthmreg] We continue to use the notational conventions introduced above. Our goal is to find an event $E$ and a constant $\gamma > 0$ such that assumption (\ref{eq:assumpfixedpoint}) of Lemma \ref{lemma:fixedpointargument} holds in $E$, with the value of $r := \Phi_P(\sF,n,\alpha,\eps)$ defined in (\ref{eq:defPhiPregression}), and a suitable $\gamma>0$. Let $x := \ln\frac{3}{\alpha} \geq \frac{1}{3}$. Following (\ref{eq:defPhiPregression}), we set
\begin{equation}\label{eq:defconstants} \theta_0:=\theta_0(\sF,P),\,\delta_\qd:=\frac{1}{32\theta_0}\mbox{ and }\delta_{\mt} = \frac{1}{448\theta_0^2}.\end{equation}
 






\begin{center}{\bf First step:}~lower bound quadratic part by Lipschitz term.\end{center}

Lemma \ref{lemma:fixedpointargument} requires control from below of $\ell_f-\ell_{f^\star_P}$, which includes a quadratic term. However, our assumptions are on the process $f-f^\star_P$ without the square. Therefore, our first step will be to find a bounded Lipschitz minorant for the quadratic term, to which we can apply concentration and contraction. Specifically, consider $a,c > 0$ with $2c > a$ (specific values to be chosen later). Define \[\psi(y) = r^2\left(2a\left(\frac{|y|}{r}\wedge c\right) - a^2\right)_+\,\,(y\in\R).\] Then:
\begin{enumerate}
\item $0\leq \psi(y)\leq y^2 $ for all $y\in\R$: this follows from the fact that the graph of $y\mapsto y^2$ is lower bounded by the tangent line at $y=ar$;
\item $\psi$ is $2ar$-Lipschitz with $\psi(0) = 0$; and
\item $\psi$ is bounded above by the constant $M_q:= r^2(2ac - a^2)$.
\end{enumerate}

Thus, for all $f\in\sF$ with $r_f=r$:
\begin{equation}\label{eq:lowerboundqdlipschitz}\Tmhat_{n, \phi n}^\varepsilon \left(\ell_f-\ell_{f^\star_P}\right)\geq \Tmhat_{n, \phi n}^\varepsilon \left(\psi(f-f^\star_P) - 2(\mt_{{f}} - P\mt_{{f}})\right),\end{equation}
where we also used (\ref{eq:1storderopt}).
\begin{center}{\bf Second step:}~define $E$ and lower bound $\Pr(E)$ via \MTthmmaster.\end{center}

Recall that our goal is to prove the existence of an event $E$ with $\Pr(E)\geq 1-\alpha$ so that, when $E$ holds, both  (\ref{eq:boundtheoremregression}) and (\ref{eq:oracle}) are satisfied. To do this, we recall the definition of $\sF_{\qd}(r) = \sF_{\qd}(r,P)$ and $\sF_{\mt}(r)=\sF_{\mt}(r,P)$ from \MTregrelparam, and set:
\begin{eqnarray}\sF_1 &:=& \psi\circ \sF_{\qd}(r)= \left\{\psi(f-f^\star_P)\,:\,f\in \sF,\,\|f-f^\star_P\|_{L^2(P_\bX)} = r\right\};\\
\sF_2 &:=& \sF_{\mt}(r) = \left\{  \mt_f - P\mt_f  : f \in \sF, ~ \|f-f_P^\star\|_{L^2(P_\bX)}\leq r \right\}.\end{eqnarray}

Now define
\begin{equation}\label{eq:defetaq}\eta_\qd := \left(6\phi+\frac{3 x}{n}\right)M_\qd +  \nu_2\left( \sF_1\right)\sqrt{\frac{2 x }{n}},\end{equation}
with $M_q$ as above, and \begin{equation}\label{eq:defetam}\eta_\mt :=  \left(6\phi+\frac{3 x}{n}\right)M_\mt + {\rm rem}_{M_\mt}\left( \sF_2\right) + \nu_2\left( \tau_{M_\mt}\circ\sF_2\right)\sqrt{\frac{2 x }{n}},\end{equation}
with $M_\mt$ soon to be defined. The event $E$ is defined as follows:
\begin{equation*} E = \left\{
\begin{gathered}
    \forall a_\qd, a_\mt \in \R, f_\qd \in \sF_\qd(r),\, f_\mt \in \sF_\mt(r) \\ \left| \Tmhat^\varepsilon_{n,\phi n} (a_\qd  \psi(f_\qd) + a_\mt f_\mt ) - a_\qd P \psi(f_\qd) \right| \leq \\
    |a_\qd| \left\{ 2\Emp_n\left( \sF_1^o \right) + \eta_\qd \right\} + |a_\mt| \left\{ 2\Emp_n\left( \tau_{M_\mt} \circ \sF_2 \right) + \eta_\mt \right\}
\end{gathered}
\right\}\end{equation*}

The event $E$ is measurable because $\sF$ has a countable dense subset. We now argue that $E$ is precisely the kind of event whose probability is bounded in \MTthmmaster. To see this, we follow the notation in that theorem, set $m:=2$, $x_1=x_2=x$ and $\sF_1$ and $\sF_2$ as above. We make the following choices of $M_i$, $b_i$ and $t_i$:

\begin{itemize}
\item The functions in $\sF_1=\psi\circ \sF_\qd(r)$ take values in $[0,M_{\qd}]$, so their expectations are also in this range. It follows that all functions in the centered class $\sF_1^o$ are bounded by $M_1:=M_{\qd}$ in absolute value. This means we can take $t_1=b_1=0$. 
\item Now consider $\sF_2 = \sF_\mt(r)$. Note that $Pf_{\mt}=0$ for all $f_{\mt}\in \sF_{\mt}(r)$, so $\sF_2=\sF_2^o$. Since $\sF_2$ may be unbounded, we will take $b_2=1$, and use \MTlemmabounds\ to obtain $M_\mt := M_2$ and $t_2 \geq x_2$ satisfying the assumptions of \MTthmmaster\ and also the bound
\begin{align}\nonumber
    2\Emp_n\left( \tau_{M_\mt} \circ \sF_2 \right)+\eta_\mt &\leq C_\eps\left\{8\Emp_n(\sF_2 ) + \inf_{q\in[1,2]}\nu_q(\sF_2) \left( \frac{\ln\frac{3}{\alpha}}{n} \right)^{1-\frac{1}{q}}\right.\\ \nonumber
    &~~~~ \left.+ \inf_{p \geq 1}\nu_p(\sF_2) \eps^{1 - \frac{1}{p}}\right\},
\end{align}
where $C_\eps := 384\left(1 + \frac{\eps}{ \eps \wedge \left(\frac{1}{2}-\eps\right) }\right)$ can be bounded by $768$ noticing that (\ref{eq:reghip0}) implies $\eps \leq \frac{1}{96\theta_0^2} < \frac{1}{4}$. Using further that $\nu_p(\sF_2) = \nu_p(\sF_\mt(r))\leq r\,\kappa_p(\sF)$, we obtain the bound:
\begin{align}\label{eq:boundLemma63}
    2\Emp_n\left( \tau_{M_\mt} \circ \sF_2 \right)+\eta_\mt &\leq 768\left\{8\Emp_n(\sF_2 ) + r\,\inf_{q\in[1,2]}\kappa_q(\sF) \left( \frac{\ln\frac{3}{\alpha}}{n} \right)^{1-\frac{1}{q}}\right.\\ \nonumber
    &~~~~ \left.+ r\,\inf_{p \geq 1}\kappa_p(\sF) \eps^{1 - \frac{1}{p}}\right\},
\end{align}
\end{itemize}
The upshot of this discussion is that \MTthmmaster\ can indeed be used to bound the probability of $E$, and we obtain:
\[\Pr(E)\geq 1-3e^{-x}\geq 1-\alpha.\]
From now on, we perform all calculations deterministically while assuming that $E$ holds. 
\begin{center}{\bf Third step:} bounds assuming $E$ holds.\end{center}
We combine the lower bound from the first step with the one defining the event $E$. Taking $a_\qd = 0$, $a_\mt = 2$ in $E$ gives, for $r_f \leq r$,
\begin{align}\label{eq:smallradius}
     2\Tmhat^\varepsilon_{n,\phi n}\left(\mt_f - P\mt_f\right) \leq   \underbrace{2\left\{2\Emp_n\left( \tau_{M_\mt} \circ \sF_2 \right)+\eta_\mt\right\}}_\text{(i)}.
\end{align}
Similarly, to consider $r_f=r$ we take $a_\qd = 1$, $a_\mt = -2$ in $E$ and obtain
\begin{align}\label{eq:largeradius}
    \Tmhat^\varepsilon_{n,\phi n}\left(\ell_{f}-\ell_{f^\star_P}\right) &\geq  \Tmhat^\varepsilon_{n,\phi n}\left( \psi(f-f^\star_P) - 2 \left(\mt_f - P\mt_f\right) \right)\\ \nonumber
    & \geq \underbrace{\inf_{f_\qd \in \sF_\qd(r)} P\psi(f_\qd)}_\text{(ii)} -\left\{ \underbrace{2\Emp_n\left( \sF_1^o \right)}_\text{(iii)} + \underbrace{\eta_\qd}_\text{(iv)} \right\} \\ \nonumber
    & ~~~~ -  \underbrace{2\left\{2\Emp_n\left( \tau_{M_\mt} \circ \sF_2 \right)+\eta_\mt\right\}}_\text{(i)},
\end{align}
where the last inequality is where we need $E$ to hold.

The bounds in (\ref{eq:largeradius}) and (\ref{eq:smallradius}) are still unwieldy. To obtain more useful bounds for (i - iv), we need a few calculations that are quite messy and not too enlightening.

\noindent\textbf{Bound (i).} Note that, since $r\geq r_\mt(\delta_\mt)$ one has $\sF_{\mt}(r)/r\subset \sF_{\mt}(r_\mt(\delta_{\mt}))/r_\mt(\delta_{\mt})$. Combining this with symmetrization and  the definition of $r_\mt(\delta_{\mt})$, we obtain:
\[\Emp(\sF_\mt(r)) \leq 2\Rad(\sF_\mt(r))\leq \frac{2r}{r_\mt(\delta_{\mt})}\Rad(\sF_\mt(r_\mt(\delta_{\mt})))\leq 2\delta_\mt\,r\,r_\mt(\delta_\mt).\]
By (\ref{eq:boundLemma63}), we obtain:
\begin{align}\label{eq:bound(i)}
    \text{(i)}&\leq 1536\,r\left\{ 16 \delta_\mt r_\mt(\delta_\mt) + \inf_{q\in[1,2]}\kappa_q(\sF) \left( \frac{x}{n} \right)^{1-\frac{1}{q}} + \inf_{p \geq 1}\kappa_p(\sF) \eps^{1 - \frac{1}{p}}\right\} \leq 14 \delta_\mt r^2,
\end{align}
where the upper bound is provided by the choice of $r = \Phi_P(\sF, \alpha, n, \eps)$ in (\ref{eq:defPhiPregression}) and the choice of $\delta_\mt$ in (\ref{eq:defconstants}).

\noindent\textbf{Bound (ii).} Using $y \wedge c \geq y - \frac{y^2}{c}$ and the definition of $\theta_0$ we can bound, for $r_f=r$,
\begin{align*}
    P \psi( f - f^\star_P) & \geq P\left\{ r^2\left(2a\left( \frac{|f-f^\star_P|}{r} - \frac{|f-f_\qd|^2}{r^2 c} \right) - a^2\right) \right\} \geq r^2 \left\{ 2a\left( \frac{1}{\theta_0} - \frac{1}{c} \right) - a^2 \right\}.
\end{align*}

\noindent\textbf{Bound (iii).} Using contraction and symmetrization (\MTsym) together with the fact that $r\geq r_\qd(\delta_\qd)$ give:
  \begin{align*}
     \Emp_n\left( \sF_1^o \right) = \Emp_n\left( \psi\circ\sF_\qd(r) \right) &\leq 2\Rad_n\left( \psi\circ\sF_\qd(r)\right) \leq 4ar\Rad_n\left( \sF_\qd(r) \right) \leq 4a \delta_\qd r^2.
 \end{align*}

\noindent\textbf{Bound (iv).} Since $\sF_1 = \psi\circ \sF_\qd(r)$, $\psi(0) = 0$ and $\psi$ is $2ar$-Lipschitz,
\[ \nu_2\left(\sF_1^o\right) \leq 2ar \, \sup_{f_\qd\in\sF_{\qd}(r)}\|f_\qd\|_{L^2(P_\bX)} = 2ar^2. \]

Observe that ${\rm rem}_{M_\qd}\left( \sF_1^o \right) = 0$ because $\sF_1^o$ is uniformly bounded by $M_\qd$, thus
\[ \eta_\qd \leq \left( 6\phi + \frac{3x}{n} \right)M_\qd + 2a r^2\sqrt{\frac{2 x }{n}}.  \]

\noindent\textbf{End of third step.} From (\ref{eq:smallradius}) and (\ref{eq:bound(i)}) we have 
\begin{equation}
    \label{eq:smallradius2}
    2\,\sup_{r_f \leq r} \Tmhat_{n, \phi n}^\varepsilon \left(\mt_{{f}} - P\mt_{{f}}\right) \leq 14 \delta_\mt r^2.
\end{equation} 

Observe that (\ref{eq:reghip0}) and our choice $x = \ln\frac{3}{\alpha}$ imply
\[ \sqrt{\frac{2x}{n}} \leq \sqrt{\frac{1}{4} \left( 6\phi + \frac{3x}{n} \right)} \leq \frac{1}{8 \theta_0},  \]
so, (\ref{eq:largeradius}) and the bounds on (i-iv) give
 \begin{align}\label{eq:largeradius2}
     \Tmhat^\varepsilon_{n,\phi n}\left( \ell_{f} - \ell_{f^\star_P} \right) &\geq r^2 \left\{ 2a\left( \frac{1}{\theta_0} - \frac{1}{c} \right) - a^2 \right\} - 8 a \delta_\qd r^2\\ \nonumber
     & ~~~~  - \left( 6\phi + \frac{3x}{n} \right)M_\qd - 2 a r^2 \sqrt{\frac{2x}{n}} - 14 \delta_\mt r^2\\ \nonumber
     \mbox{((\ref{eq:reghip0}) $+$ $M_q\leq 2acr^2$)}& \geq r^2 \left\{ 2a\left( \frac{1}{\theta_0} - \frac{1}{c} - \frac{c}{16 \theta_0^2} - \frac{1}{8\theta_0} - 4\delta_\qd \right) - a^2 \right\} - 14 \delta_\mt r^2.
 \end{align}
 \begin{center}{\bf Final step:}{ apply Lemma \ref{lemma:fixedpointargument} via choices of constants.}\end{center}
We finish the proof via an application of Lemma \ref{lemma:fixedpointargument}, assuming as before that $E$ holds. Recall $r:=\Phi_P$ and take $\gamma :=(32\theta_0^2)^{-1}r^2$. We defined $\delta_{\mt}=(448\theta_0^2)^{-1}$ in (\ref{eq:defconstants}); therefore, (\ref{eq:smallradius2}) gives condition (\ref{eq:assumpfixedpoint}) for $r_f\leq r$. Now consider the lower bound for $\Tmhat^\varepsilon_{n,\phi n}\left( \ell_{f} - \ell_{f^\star_P} \right)$ when $r_f = r$. Recall from (\ref{eq:defconstants}) that $\delta_{\qd}=(32\theta_0)^{-1}$. Insert this into the RHS of (\ref{eq:largeradius2}) and optimize over $2c > a > 0$. This leads to the choices $c = 4\theta_0$ and $a = c^{-1}$, giving
\[ r^2\,\sup_{a,c > 0: ~ 2c > a} 2a\left( \frac{1}{\theta_0} - \frac{1}{c} - \frac{c}{16 \theta_0^2} - \frac{1}{8\theta_0} - 4\delta_\qd \right) - a^2 = \frac{r^2}{16\theta_0^2} = 2\gamma. \]
Finally, combine (\ref{eq:largeradius2}) and (\ref{eq:smallradius}) to obtain:
\begin{align*}\label{eq:lowerboundlarger}
    \mbox{when }r_f = r: \, \Tmhat^\varepsilon_{n,\phi n}\left( \ell_{f} - \ell_{f^\star_P} \right) \geq \gamma.
\end{align*}
This gives the missing half of (\ref{eq:assumpfixedpoint}). Therefore, Lemma \ref{lemma:fixedpointargument} may be applied, and this finishes the proof.\end{proof}

\section{The relation between contamination level and the small-ball assumption}\label{sec:sm_12}
This section corresponds to \MTremarksm, where we note that a restriction of the form $\eps\leq c\,\theta_0(\sF,P)^{-2}$ is necessary in the setting of robust regression with quadratic loss (as in \MTthmreg).

To prove this, we use a family of distributions and a contamination model discussed in \MTsetupB. Given a dimension $d\in\N$ and a parameter $p\in (0,1)$, let $X'\sim \texttt{Normal}(0_{\R^d},I_{d\times d})$, $\xi\sim \texttt{Normal}(0,1)$ be independent. Given $\beta\in \R^d$, we let $P_{\beta}$ denote the distribution of the random pair $(X,Y)$ given by  
\[X= {B_i}\,\frac{X'}{\sqrt{p}}\mbox{ and }Y=\langle X,\beta\rangle + \xi.\]

Because $X$ is isotropic, robust linear regression in this setting consists of estimating $\beta$ in the Euclidean norm from an $\eps$-contaminated i.i.d. sample from $P_\beta$. Now, clearly, 
\[\forall \beta\in\R^d\,:\,\theta_0(\sF,P_\beta)=\sqrt{\frac{\pi}{2p}}.\]
The next Lemma roughly says that a contaminated sample from $P_\beta$ with $\eps>c\,\theta_0(\sF,P)^{-2} $ essentially contains no information about $\beta$. More precisely, the Lemma implies via standard arguments that for any $R>0$, one can find at least one $\beta\in\R^d$ for which the error of any estimator for $\beta$ will be larger than $R$, with probability $\geq 1-\alpha$.

\begin{lemma}\label{lem:coupleverydifferentthings} For any $\beta\in\R^d$, and parameters $n\in\N$, $\eps,p,\alpha\in (0,1)$ satisfying
\[\eps \geq 2p = \frac{\theta_0(\sF,P_{\beta})^2}{\pi}\mbox{ and }n \geq \left(\frac{2(1-p)+2\eps}{p}\right)\,\ln\frac{1}{\alpha}\]
one can define $\eps$-contaminated samples $Z_{1:n}^{\eps,\beta}$ from $P_\beta$ and an i.i.d. (uncontaminated) sample $Z_{1:n}^{0,0_{\R^d}}$ from $P_{0_{\R^d}}$ such that
\[\Pr\left\{Z_{1:n}^{\eps,\beta}=Z_{1:n}^{0,0_{\R^d}}\right\}\geq 1-\alpha.\]
\end{lemma}
\begin{proof}
Let $Z^{\beta}_{1:n}\stackrel{i.i.d.}{\sim} P_\beta$, with each $Z^\beta_i=(X_i,Y_i)$. We may assume that $X_i =B_i\,X'_i/\sqrt{p}$ and $Y_i = \langle X_i,\beta\rangle +\xi_i$ with $(B_i,X'_i,\xi)\sim (B,X',\xi)$ as above. As a result, a $\texttt{Ber}(p)$ proportion of $X_i$ in the sample are non-zero. To define an $\eps$-contaminated sample from $P_\beta$, we choose a subset $\mathcal{O}\subset \{i\in [n]\,:\, B_i\neq 0\}$ of size $\# \mathcal{O}\leq \eps\,n$ that is as large as possible, and then set:
\[Z_{i}^{\eps,\beta}:=\left\{\begin{array}{ll} (X_i,Y_i), & i\in [n]\backslash\mathcal{O}\\ (0_{\R^d},\xi_i), & i \in\mathcal{O}.\end{array}\right.\]

Now, the random sample consisting of the points $Z_i^{0,0_{\R^d}}:=(0_{\R^d},\xi_i)$ ($i\in [n]$) is i.i.d. from $P_0$. Moreover, by our definition of the contamination,
\[\left\{Z_{1:n}^{\eps,\beta}=Z_{1:n}^{\eps,0_{\R^d}}\right\}\supset \left\{\mathcal{O}= \{i\in [n]\,:\, B_i\neq 0\}\right\} = \left\{\sum_{i=1}^nB_i\leq \eps n\right\}.\]

The sum $\sum_{i=1}^nB_i$ is a binomial random variable with mean $pn$. Chernoff bounds give
\[ \Pr\left\{ \sum_{i=1}^nB_i\leq \eps n \right\} \geq 1 - e^{-\frac{(\eps - p)^2 n}{2\,p(1-p) + 2\eps}} \geq 1-\alpha\] where the our assumptions on $\eps$ and $n$ were used in the last two inequalities.\end{proof}

\section{Algorithms for robust linear regression: details}\label{sec:sm_13}
Recall that the estimator $\hat{f}^\eps_n = F_n(Z_{1:n}^\eps)$ is defined by the min-max problem 
\begin{equation}
    \label{eq:minimaxestimator}
    F_n(z_{1:n})\in\argmin_{f\in\sF}\left(\max_{g\in\sF}\Tmhat_{n,k}\,(\ell_f-\ell_g,z_{1:n})\right).
\end{equation}

The difference $\ell_f - \ell_g$ is convex in $f$ and concave in $g$, which suggests that the min-max problem above can be efficiently solved. Unfortunately, the trimmed mean operation $\Tmhat_{n,k}$ introduces complications; e.g. it is not obvious how to compute sub- and super-gradients. Besides that, the TM estimator as discussed in the main text requires knowledge of the (usually unknown) contamination level $\eps$ to choose the trimming parameter $k = \lfloor \phi n \rfloor$. Analogous issues can also be raised for the Median of Means. Thus, such estimators can only be of practical use if one can surpass these issues. 

In the remainder of this section, we restrict ourselves to the setting of linear regression described in \MTalg.  We present heuristics to
\begin{itemize}
    \item compute the min-max problem \eqref{eq:minimaxestimator};
    \item select the trimming level $k$ (and its analogous number of buckets $K$ for the MoM).
\end{itemize}

As in the aforementioned section, we let $\ell_{\beta}$ denote the loss associated with a vector $\beta\in\R^d$. Given two vectors $\beta_m,\beta_M\in\R^d$, 
\begin{equation}
    \label{eq:tmactive}
    \widehat{T}_{n,k}\left( \ell_{\beta^m} - \ell_{\beta^M} \right) = \frac{1}{n-2k} \sum_{i\in I_k\left(\beta^m, \beta^M,z_{1:n}\right)} \left(\langle x_i,\beta^m\rangle - y_i\right)^2 - \left(\langle x_i,\beta^M\rangle - y_i\right)^2
\end{equation}
where
\begin{itemize}
    \item each data
point $z_i=(x_i,y_i)\in \R^d\times \R$, and
\item $I_k(\beta^m,\beta^M,z_{1:n})$ -- called the {\em active set for $(\beta^m,\beta^M,z_{1:n})$} -- is the set of indices $i\in [n]$ that appear in the trimmed mean, ie. the set obtained once the $k$ largest and $k$ smallest values of $[(\langle x_i,\beta_m\rangle - y_i)^2 - (\langle x_i,\beta_M\rangle - y_i)^2]$ are removed (with ties broken arbitrarily).\end{itemize}

In what follows we simplify the notation defining, for every $I \subset [n]$, $x_I$ as the $|I| \times d$ matrix with lines $x_i$, $i \in I$, with $y_I$ defined analogously. We proceed to discuss heuristics to compute the regression vectors.

\subsection{Optimization methods for a fixed trimming parameter}

We consider two algorithms to evaluate \eqref{eq:minimaxestimator}. Algorithm \ref{alg:plugin} is the Plug-in method considered in the main text. Algorithm \ref{alg:admm} is an Alternating Direction Method of Multipliers (ADMM) which is an adaptation of the best performing method in \cite{Lecue2020}.

\begin{algorithm}[ht!]
\SetKwInOut{Input}{input}
\SetKwInOut{Output}{output}
\Input{$(x_1,y_1), \cdots, (x_n,y_n) \in \R^d\times \R$: the data\\$\beta_0^m, \beta_0^M \in \R^d$: a initial guess \\$\phi$: trimming ratio\\$T_{\max}$: number of iterations\\$\rho$: multiplier parameter}
\Output{$\beta^\star_{\phi}$: an approximate solution for the min-max problem}

\BlankLine

$t \leftarrow 0$

$k \leftarrow \lfloor \phi n \rfloor $ (evaluate trimming level)

\While{$t \leq T_{\max}$}{    
    $I_t^m \leftarrow I_k\left(\beta^m_t,\beta^M_t ,z_{1:n}\right)$ (get active indices)
        
    $\beta_{t+1}^m \leftarrow \left( x_{I^m_t}^T x_{I^m_t} + \rho I_d \right)^{-1}\left( x_{I^m_t}^T y_{I^m_t} + \rho \beta_t^m \right)$
    
    \BlankLine
    
    $I_t^M \leftarrow I_k\left(\beta^m_{t+1},\beta^M_t ,z_{1:n}\right)$ (get updated active indices)
    
    $\beta_{t+1}^M \leftarrow \left( x_{I^M_t}^T x_{I^M_t} + \rho I_d \right)^{-1}\left( x_{I^M_t}^T y_{I^M_t} + \rho \beta_t^M \right)$
    
    \BlankLine
    
    
    $t \leftarrow t+1$
    
}

$\beta^\star_{\phi}\leftarrow \argmin\left\{\Tmhat_{n,k}(\ell_\beta, z_{1:n})\,:\,\beta \in \bigcup_{t=1}^{T_{\max}}\{\beta^m_t, \beta^M_t\}\right\}.$

\caption{Alternating Direction Method of Multipliers.}\label{alg:admm}
\end{algorithm}

For the Plug-in method, we set $T_{\max} = 20$ in all experiments, as solutions do not improve beyond this number iterations. For the same reasons, we set $T_{\max}=50$ for ADMM. The parameter $\rho$ in Algorithm \ref{alg:admm} is set to $5$ as in \cite{Lecue2020}.

\begin{remark}Algorithm \ref{alg:admm} differs from its analogue in \cite{Lecue2020} in three ways. First, their estimator is based on the MoM principle. Second, that paper considers a sparse regression setting with an $\ell_1$ penalty. Third, the final output $\beta^\star_\phi$ in \cite{Lecue2020} is simply the final iterate $\beta^m_{T_{\max}}$. Preliminary experiments show that choosing $\beta^\star_\phi$ as a minimizer of the trimmed empirical risk improves in all settings considered in what follows.\end{remark}

\subsection{Cross-validation} We first recall the cross-validation procedure from the main text. Assume $\phi_1 \leq \phi_2 \leq \dots \leq \phi_m$ is a grid of possible choices for the trimming ratio $\phi$. Recall that the trimming ratio $\phi$ and the trimming level $k$ are related by $k=\lfloor \phi n \rfloor$. Let $v \leq n$ be the number of folds: that is, $[n]$ is partitioned into sets $\{B_l\}_{l=1}^v$ with respective sizes $n_l := |B_l| \in \{\lfloor n/v\rfloor,\lfloor n/v\rfloor+1\}$. The procedure works as follows. 

\begin{enumerate}
    \item For each choice of $(j,l)\in [m]\times [v]$, let $\beta^\star_{\phi_j}([n] - B_l)$ be the output of Algorithm \ref{alg:plugin} on the $n-n_l$ data points $(x_i,y_i)_{i\not\in B_l}$.
    \item For each choice of $(j,l)\in [m]\times [v]$, estimate the loss of $\beta^\star_{\phi_j}([n] - B_l)$ via a trimmed mean with ratio $\phi_j$ on the fold $B_l$:
    \[L(j,l):=\Tmhat_{n_l, \phi_j n_l} \left(\ell_{f_{j,l}} (z_{i})_{i\in B_l}\right).\]
    \item For each $j\in [m]$, associate a loss with trimming ratio $\phi_j$ via \[L(j):={\rm median}(L(j,l)\,:\,l\in [v]).\]
    \item Choose $\phi^\star = \phi_{j^\star}$ where $j^\star = \argmax_{j = 2, \dots, m} \frac{L(j-1)}{L(j)}$.
    \item Compute the final estimator $\beta^\star_{\phi^\star}$ by running Algorithm \ref{alg:plugin} with trimming ratio $\phi^\star$ on the full dataset $(z_i)_{i\in[n]}$. 
\end{enumerate}

We here call the choice of $\phi^\star$ made using step (4) above as choice by slope maximization (abbr. \textbf{max slope}). We also test a variant of step (4), that is used in \cite{Lecue2020}:
\begin{primenumerate}
    \item Choose $\phi^\star = \phi_{j^\star}$ where $j^\star = \argmin_{j = 1, \dots, m} L(j)$.
\end{primenumerate}

We call this last variant choice by loss minimization (abbr. \textbf{min loss}). Both approaches will be compared in the next section.

\subsection{Variants of MoM} As noted in the main text, we compare the performance of the trimmed mean estimator against a variant of the Median of Means procedure from \cite{Lecue2020}. Recall that MoM requires splitting the $n$ data points into $K$ buckets of approximately equal size. These splits are performed randomly at each iteration, as recommended by \cite{Lecue2020}. 

As observed in the article, the parameter $K$ is a close analogue of the trimming parameter $k$ in our procedure. This allows us to adapt the optimization and cross validation procedures described above. 

\begin{itemize}
    \item To adapt the Plug-in and ADMM methods, we use a different concept of active set. Consider the blocks $A_1,\dots,A_K\subset [n]$ used by the median-of-means construction. Then the active set $I_k(\beta^m,\beta^M)$ is the block of indices $A_r$ for which
    \begin{align*}
        &\frac{1}{\# A_r}\sum_{i\in A_r}(\ell_{\beta^m}(x_i,y_i) - \ell_{\beta^M}(x_i,y_i))\\
        & = {\rm median}\left\{\frac{1}{\# A_s}\sum_{i\in A_s}(\ell_{\beta^m}(x_i,y_i) - \ell_{\beta^M}(x_i,y_i))\,:\, s\in [K]\right\},
    \end{align*}
    with ties between blocks broken arbitrarily.
    \item When performing the optimization iterations, the blocks of median-of-means are resampled uniformly at random at each step.
    \item The cross validation procedure is now over choices of $K_j$. However, the loss estimate $L(j,l)$ are performed via a MoM estimator using the data in fold $B_l$ with $K_j/v$ blocks. With these changes, both cross-validation methods may be applied. 
\end{itemize}

\subsection{Details on Huber regression} For a given $L > 0$, Huber regression consists in solving
\begin{equation}
\label{eq:huber_argmin}
    \argmin_{\beta \in \R^d} \frac{1}{n}\sum_{i=1}^n H_L\left( \langle x_i, \beta \rangle - y_i \right),
\end{equation}
where
\[ H_L(t) = \begin{cases}
    2 L |t| - L^2 \text{ if } |t| > L\\
    t^2 \text{ if }|t| \leq L
\end{cases}.\]
Let $M \geq 1$ be called shape parameter and $\sigma > 0$ be a scale parameter. Taking $L = M\sigma$, the optimization procedure \eqref{eq:huber_argmin} is equivalent to
\[ \argmin_{\beta \in \R^d} \frac{1}{n}\sum_{i=1}^n H_M\left( \frac{ \langle x_i, \beta \rangle - y_i }{\sigma} \right). \]
If $\sigma$ scales with the sample, this parametrization let us uncouple the dependence of $L$ on the sample and on the desired robustness. Indeed, given only the shape parameter $M$, the scaling parameter can be learned from the data solving
\begin{equation}
    \label{eq:huber_sigma}
    \argmin_{\beta \in \R^d, \sigma > 0} \sigma + \frac{\sigma}{n}\sum_{i=1}^n H_M\left( \frac{\langle x_i, \beta \rangle - y_i }{ \sigma } \right),
\end{equation}
as proposed by \cite{huber2011robust}. Since this procedure is already implemented in \texttt{sklearn} Huber regression method, we simply use it as available there.

In this case, the cross validation procedure selects a shape parameter $M$ between $1$ and $1.35$. The quantity $1.35$ is not incidental, this value yields a $95\%$ statistical efficiency on a Gaussian sample. Thus, as $M$ grows, Huber regression favors efficiency over robustness. We also consider the two variants of the cross validation procedure, but letting
\begin{equation}
    \label{eq:Ljl_huber}
    L(j,l) = \frac{1}{|B_l|} \sum_{i \in B_l} | d_i | \mathbf{1}_{|d_i| > L_j} + d_i^2 \mathbf{1}_{|d_i| \leq L_j}
\end{equation}
where $L_j = M_j\sigma(M_j)$, $\sigma(M_j)$ is obtained from \eqref{eq:huber_sigma} and $d_i = \langle x_i, \beta \rangle - y_i$.

\begin{remark}
    Notice that $L(j,l)$ as in \eqref{eq:Ljl_huber} does not correspond to the Huber loss \eqref{eq:huber_argmin} or its penalized version \eqref{eq:huber_sigma}. Our prior experiments have shown that the choices \eqref{eq:huber_argmin} and \eqref{eq:huber_sigma} tend to select only $M=1.35$ or $M=1.0$ for both the \textbf{min. loss} and the \textbf{max. slope} strategies. Meanwhile, \eqref{eq:Ljl_huber} yields more nuanced choices of $M$ and better performance. 
\end{remark}

\section{Experiments on linear regression: details and additional results}\label{sec:sm_14}

In this section, we present a detailed description of our experiments with trimmed-mean-based-regression, median-of-means-based-regression, Huber regression, quantile regression, and ordinary least squares. In particular, we will also discuss in greater detail the experimental results from the main text. For convenience, we repeat some of the material from the main text. A GitHub repository with code to reproduce all figures and experiments in the paper can be found at
\begin{center}
\texttt{github.com/lucasresenderc/trimmedmean}.
\end{center}

\subsection{Extended Setup A} We extend Setup A from the main text to include heteroscedastic errors. This gives us many examples on which to test the performance of Algorithms \ref{alg:plugin} and \ref{alg:admm} for TM and MoM, as well as the two cross validation strategies proposed for TM, MoM and Huber regressions.

As in the main text, we consider a linear model. Let $d\geq 1$ be an integer and $X_1, X_2, \cdots, X_n$ be i.i.d. standard Gaussian's. For $i\in [n]$, define \[Y_i = \langle X_i , \beta^\star \rangle + \xi_i \; \mbox{ where }\beta^\star = \frac{1}{\sqrt{d}}\left[1,1,\cdots,1\right] \in \R^d\] and the $\xi_i$ are errors. To obtain $\xi_i$ we use the skewed generalized t-distribution from \cite{lian2025novel} to sample $\eta_i$. This distribution allows us to control the moments and the skew by changing the parameters $\alpha > 0$ and $\lambda \in (-1,1)$. Recall that if $\lambda = 0$, $\eta_i \sim t(2\alpha)$, where $t(\infty)$ denotes the standard Gaussian. On top of that, we consider two scenarios:
\begin{enumerate}
\item Homoscedastic: simply set $\xi_i=\eta_i$ for each $i\in [n]$.
\item Heteroscedastic: in this case we take $\xi_i=\,\exp(\|X_i\|^2)/2)\,\eta_i$ for each $i\in[n]$.
\end{enumerate}

In our experiments we let $\alpha \in \{0.5, 1, 2, \infty\}$, $\lambda \in \{ 0, .3, .6, .9 \}$ and consider both homoscedastic and heteroscedastic errors, yielding $32$ different combinations. We do not consider $\lambda < 0$ as they are symmetric to the corresponding $|\lambda|$.

The contamination model is defined as in the main text: a set of indices $\sO \subset [n]$ of size $\lfloor\eps n\rfloor$ is chosen uniformly at random, and one then sets
\[ (X_i^\eps, Y_i^\eps) = (X_i, Y_i) ~ \forall i\not\in\sO \text{ and } (X_i^\eps, Y_i^\eps) = (\beta^\star, 10000) ~ \forall i\in\sO. \]

\begin{remark}Before proceeding, notice that our setup is such that the $L^2$ error $\|\langle \beta,\cdot\rangle - \langle \beta^\star,\rangle\|_{L^2(P_{\bX})}$ equals the Euclidean distance between $\beta$ and $\beta^\star$.\end{remark}

\subsection{Experiment design}

We design an experiment to simultaneously compare
\begin{itemize}
    \item the performance of the Plug-in and the ADMM algorithms for the TM and the MoM based regressions;
    \item the two cross-validation parameter selection variations (max slope and min loss) both in terms of its ability to estimate the contamination level and in terms of the error obtained;
    \item the overall performance of all five regression methods.
\end{itemize}

As noted above, the different choices for the $\xi_{1:n}$ lead to 32 different possibilities for the uncontaminated data distribution. As in the main text, we vary the contamination level $\eps$ in
\[ \left\{ 
0, \frac{2}{100}, \frac{4}{100}, \frac{6}{100}, \frac{8}{100}, \frac{10}{100}, \frac{15}{100}, \frac{20}{100}, \frac{30}{100},\frac{40}{100}  \right\}.\]
Thus giving $160$ combinations of data distributions and contamination levels. In all cases, we set $d=20$, $n=300$, and we evaluate performance by performing 96 independent trials. Moreover, cross-validation with always be performed with $v=5$ folds.

For each trial, we evaluated:
\begin{itemize}
    \item For the TM and the MoM: the performance of the four different combinations of optimization and cross validation methods (ADMM with max slope, Plug-in with max slope, ADMM with min loss and Plug-in with min loss).
    \item For the Huber regression: the performance of the two cross validation strategies.
    \item For quantile and OLS no parameter was necessary, so we simply evaluated their performance.
\end{itemize}

The grid of values for $\phi$ used during the cross-validation for the trimmed-mean-based-regression was
\[ \left\{ \eps' + \frac{1}{30} : \eps' \in \left\{ 
0, \frac{2}{100}, \frac{4}{100}, \frac{6}{100}, \frac{8}{100}, \frac{10}{100}, \frac{15}{100}, \frac{20}{100}, \frac{30}{100},\frac{40}{100}  \right\} \right\}, \]
for the selection of the number of buckets $K$ for the median-of-means-based-regression was
\[ \left\{ 2\left(\eps' + \frac{1}{30}\right) + 1 : \eps' \in \left\{ 
0, \frac{2}{100}, \frac{4}{100}, \frac{6}{100}, \frac{8}{100}, \frac{10}{100}, \frac{15}{100}, \frac{20}{100}, \frac{30}{100},\frac{40}{100}  \right\} \right\}, \]
and for the shape parameter $M$ for Huber regression was
\[ \left\{ 1 + \frac{i}{9} \times 0.35 : i = 0,1,2,\dots, 9 \right\}. \]

\subsection{Preliminary experiments on algorithms and cross validation strategies}

We begin observing that over all executions, the Plug-in method outperformed the ADMM $52.7\%$ for of times for the TM and $92.7\%$ of the times for the MoM. The Plug-in algorithm is also faster to run. This justifies our choice for Algorithm \ref{alg:plugin}.

\begin{figure}
    \centering
    \includegraphics[width=\linewidth]{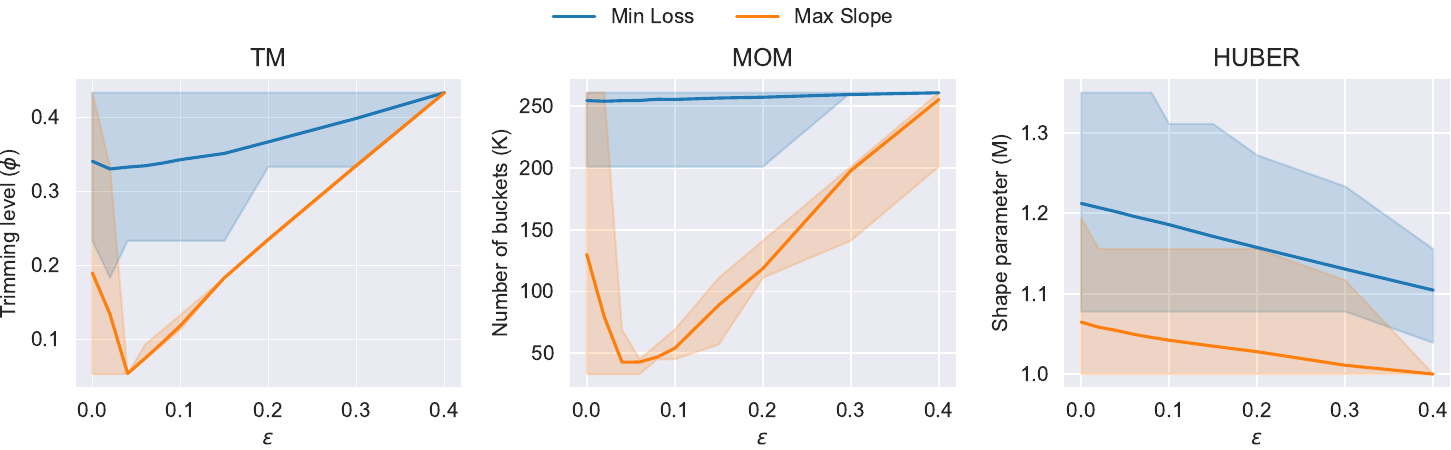}
    \caption{Comparison of the two cross validation strategies -- minimum loss (\textbf{min. loss.}) and maximum slope (\textbf{max. slope.}) -- across all experimental results. Bands represent the $5\%$ and the $95\%$ percentiles, while the solid line displays the median.}
    \label{fig:cross_validation}
\end{figure}

We also considered the performance of the two cross validation strategies proposed: the one based on the slope heuristic and the one based on the minimum loss. Figure \ref{fig:cross_validation} provides a comparison between both strategies over all our experimental results. We notice that the \textbf{min. loss.} strategy tends to select larger trimming levels for the TM and larger number of buckets for the MoM, trying to approximate both strategies to the median. Meanwhile, the parameter obtained with the \textbf{max. slope} strategy linearly follows the contamination level $\eps$.  We do note that both cross-validation methods are conservative for small contamination levels. As for the Huber regression, the \textbf{min. loss.} strategy provides a wider range of choices for the shape parameter. For instance, it is able to select $M=1.35$ when $\eps = 0$ and $\alpha = \infty$. Moreover, it also outperforms the \textbf{max. slope} in terms of $L_2$ error. Therefore, we choose the \textbf{max. slope} strategy for TM and MoM and the \textbf{min. loss} strategy for Huber regression.

\subsection{Experimental results} We now show the results of all the experimental combinations described above.

\begin{figure}[ht!]
    \centering
    \includegraphics[width=\linewidth]{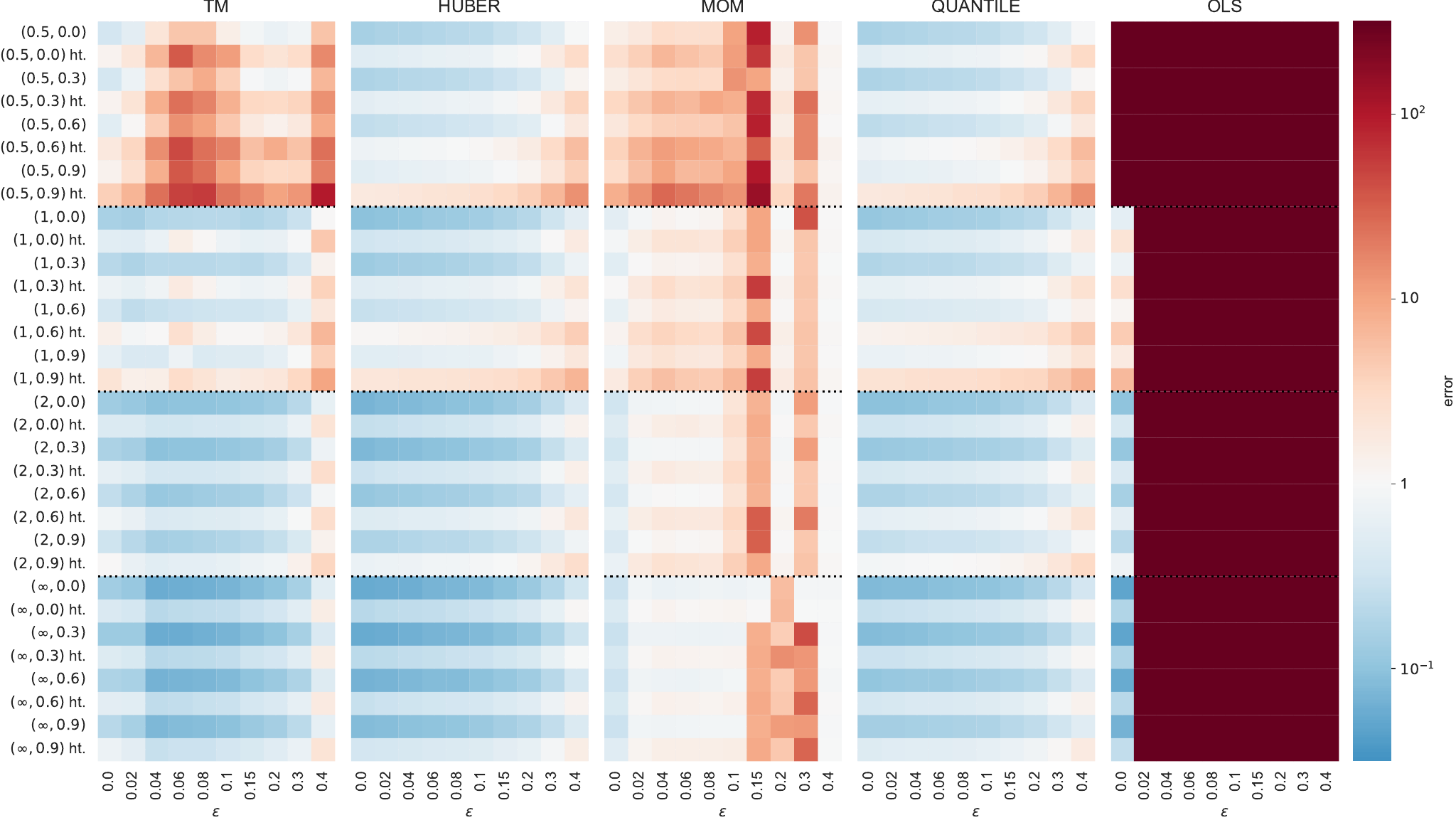}
    \caption{Heatmap of the $L^2$ error $\| \widehat{\beta}_n - \beta^\star \|$. Each line is a different combination of $(\alpha, \lambda)$ and homoscedasticity/heteroscedasticity. The columns vary the method and the contamination level $\eps$.}
    \label{fig:error_plots}
\end{figure}

\begin{figure}[ht!]
    \centering
    \includegraphics[width=\linewidth]{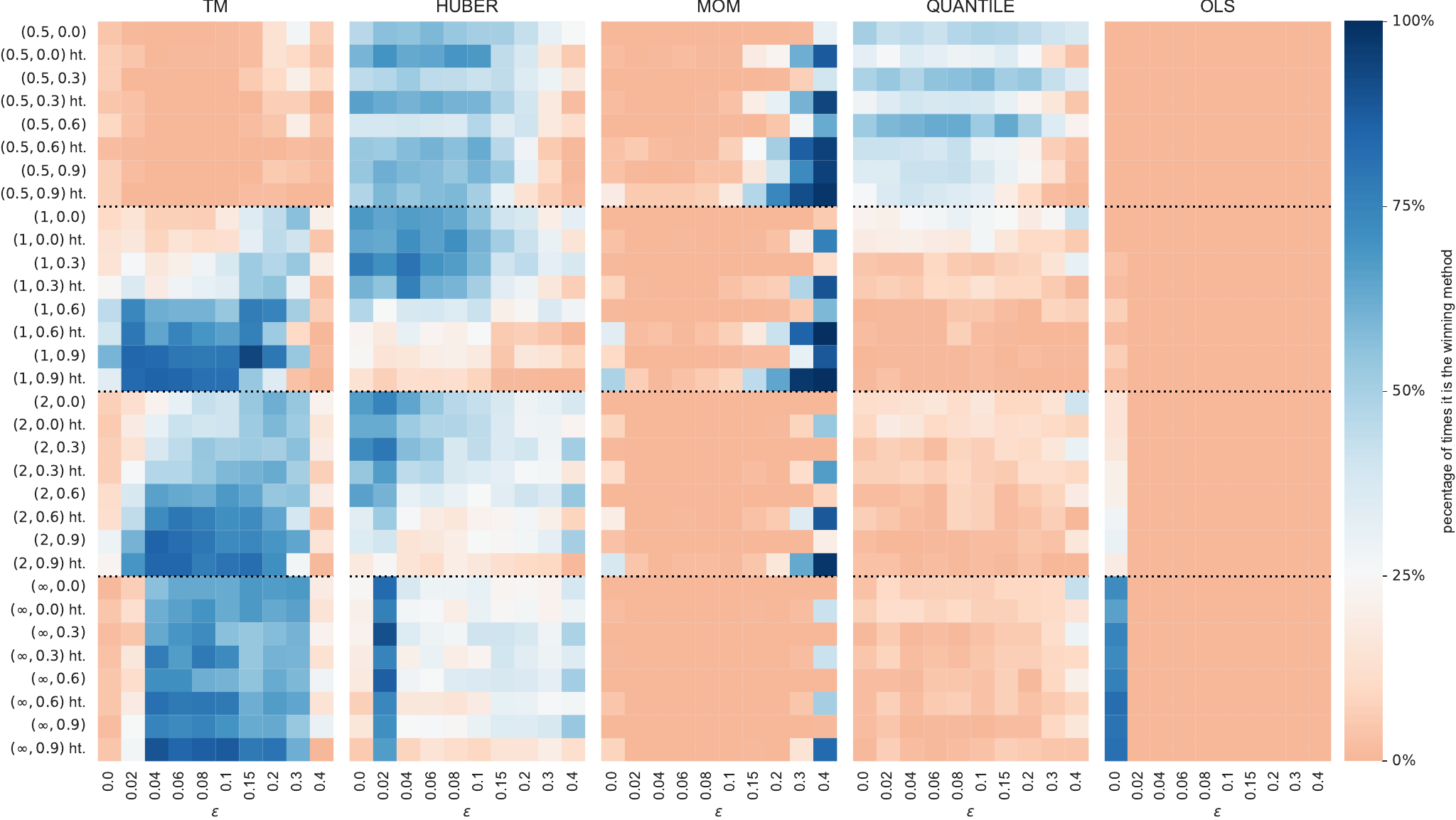}
    \caption{Percentage of seeds over which each method is the top performer. Each line is a different combination of $(\alpha, \lambda)$ and homoscedasticity/heteroscedasticity. The columns vary the method and the contamination level $\eps$. The bluer the entry, the higher the percentage of times that the specific method was the top performer for that choice of $(\alpha,\lambda,\varepsilon)$.} 
    \label{fig:winning_plots}
\end{figure}

Figures \ref{fig:error_plots} and  \ref{fig:winning_plots} correspond to two different evaluations of the method. In Figure \ref{fig:error_plots}, we see the magnitude of the error of each error, for each combination of $(\alpha,\lambda,\varepsilon)$ (blue is smallest, red is largest). The colors in Figure \ref{fig:winning_plots} corresponds to the percentage of times a specific method was the top performer for the corresponding combination of $(\alpha,\lambda,\varepsilon)$. 

We highlight the conclusions from the main text:
\begin{itemize}
    \item Huber regression is the most well-rounded method, being an overall good choice.
    \item MoM is the best performing method in when the contamination level is very large ($\eps = 0.4$) and the distributions are very heavy tailed. Quantile regression tends to outperform it in general.
    \item The TM outperforms Huber regression when tails are not extremely heavy, and deals better with skewed distributions.
    \item The OLS only performs well under no contamination and light-tailed distributions.
\end{itemize}

\begin{figure}[ht!]
    \centering
    \includegraphics[width=\linewidth]{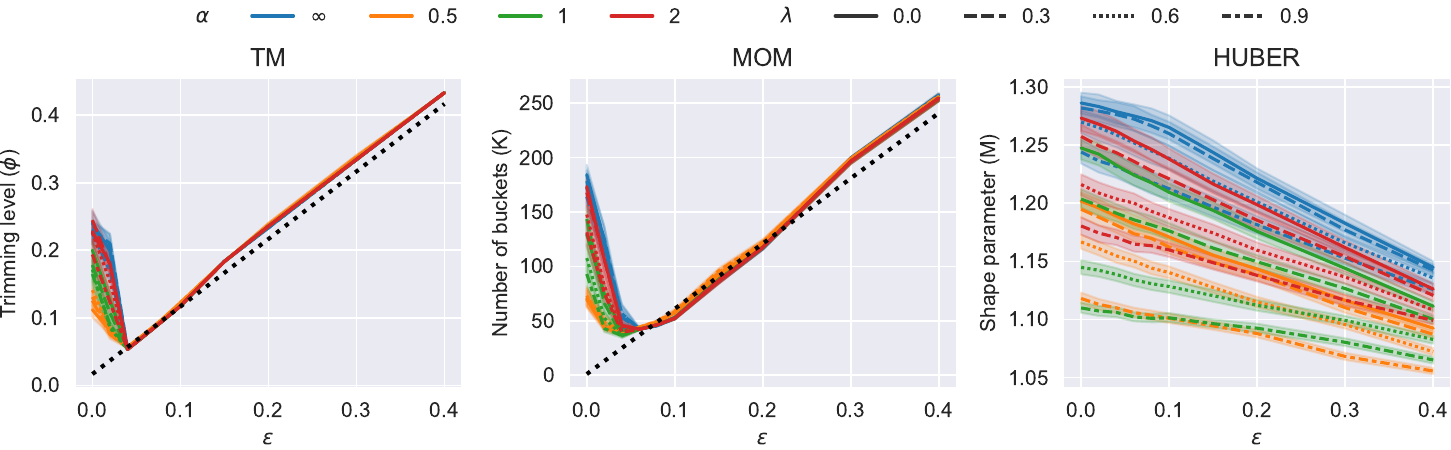}
    \caption{Parameter selected via cross-validation for different choices of $(\alpha, \lambda)$. Recall TM and MoM use the \textbf{max. slope} strategy and Huber regression uses \textbf{min. loss} strategy. Homocedastic and heteroscedastic cases are displayed together as it does not seem to impact the parameter selection.}
    \label{fig:best_param}
\end{figure}

Figure \ref{fig:best_param} displays the parameter selected by the cross validation procedure for each method. It highlights a difference in the mechanism underlying the classical Huber regression and the Median of Means or Trimmed mean. The trimming level ($\phi$) of the TM and the number of buckets ($K$) of the MoM are driven mostly by the contamination level, in a linear way. This empirical observation is aligned with our theoretical results for the TM and the ones in \cite{Lecue2020} for the MoM, since we propose $\phi \approx \frac{1}{n}\ln\frac{1}{\delta} + \eps$ and they propose $K \approx \ln\frac{1}{\delta} + 2\eps n$. Meanwhile, the shape parameter for Huber regression is driven not only by the contamination level, but also by the moments available and the skewness: it decays as $\eps$ increases, but also as $\alpha$ decreases and $\lambda$ increases. We highlight homocedasticity and heteroscedasticity have a minor impact on the selected parameters for all three methods.

\end{appendix}

\section*{Acknowledgments}
The work of RIO was supported by a Bolsa de Produtividade em Pesquisa
and a Projeto Universal from CNPq, Brazil; and by a Cientista do Nosso Estado grant from FAPERJ, Rio de Janeiro, Brazil. The work of LR was supported by the Conselho Nacional
de Desenvolvimento Científico e Tecnológico (CNPq).

\bibliographystyle{unsrtnat}
\bibliography{references}  






\end{document}